\documentclass[12pt]{article}
\usepackage{amsmath}
\usepackage{amsmath,amssymb, amsthm, latexsym, amsfonts, epsfig, color,graphicx,}
\usepackage{mathrsfs}
\usepackage{indentfirst}
\usepackage{chngcntr}
\usepackage{bbm}
\usepackage{newtxmath} 
\usepackage[colorlinks=True,linkcolor=blue,anchorcolor=blue,citecolor=blue,CJKbookmarks=true]{hyperref}
\usepackage{geometry}
\begin{document}
\renewcommand{\a}{\alpha}
\newcommand{\D}{\Delta}
\newcommand{\ddt}{\frac{d}{dt}}
\counterwithin{equation}{section}
\newcommand{\e}{\epsilon}
\newcommand{\eps}{\varepsilon}
\newtheorem{theorem}{Theorem}[section]
\newtheorem{proposition}{Proposition}[section]
\newtheorem{lemma}[proposition]{Lemma}
\newtheorem{remark}{Remark}[section]
\newtheorem{example}{Example}[section]
\newtheorem{definition}{Definition}[section]
\newtheorem{corollary}{Corollary}[section]
\makeatletter
\newcommand{\rmnum}[1]{\romannumeral #1}
\newcommand{\Rmnum}[1]{\expandafter\@slowromancap\romannumeral #1@}
\makeatother

\title{Almost everywhere convergence of Bochner-Riesz means for the Hermite type Laguerre expansions}
\author{Longben Wei, Zhiwen Duan\footnote{Corresponding author}\\
{\small {\it  School of Mathematics and Statistics, Huazhong University of Science }}\\
{\small {\it and Technology, Wuhan, {\rm 430074,} P.R.China}}  \\
{\small {\it Email:longbenwei51@gmail.com}}(L. Wei)\\
{\small {\it Email:duanzhw@hust.edu.cn}(Z. Duan)}
}
\date{}
\maketitle
\centerline{\bf Abstract }
 Consider the space $\mathbb{R}_+^d=(0,\infty)^d$ equipped with Euclidean distance and the Lebesgue measure. For every $\alpha=(\alpha_1,...,\alpha_d)\in[-1/2,\infty)^d$, we consider the Hermite-Laguerre operator $\mathcal{L}^\alpha=-\Delta+\arrowvert x\arrowvert^2+\sum_{i=1}^{d}(\alpha_j^2-\frac{1}{4})\frac{1}{x_i^2}$. In this paper we study almost everywhere convergence of the Bochner-Riesz means associated with $\mathcal{L}^\alpha$ which is defined as $S_R^{\lambda}(\mathcal{L}^\alpha)f(x)=\sum_{n=0}^{\infty}(1-\frac{4n+2\arrowvert\alpha\arrowvert_1+2d}{R^2})_{+}^{\lambda}\mathcal{P}_nf(x)$. Here $\mathcal{P}_nf(x)$ is the n-th Laguerre spectral projection operator and $\arrowvert\alpha\arrowvert_1$ denotes $\sum_{i=1}^{d}\alpha_i$. For $2\leq p<\infty$, we prove that 
 \[
 \lim_{R \to \infty} S_R^{\lambda}(\mathcal{L}^\alpha)f = f \quad \text{a.e.}
 \]
 for all $f\in L^p({\mathbb{R}_+^d})$ provided that $\lambda>\lambda(p)/2$ and $\lambda(p)=\max\{d(1/2-1/p)-1/2,0\}$. Conversely, we show that the convergence generally fails if $\lambda<\lambda(p)/2$ in the sense that there exists $f\in L^p({\mathbb{R}_+^d})$ for $2d/(d-1)< p$ such that the convergence fails. 
\section{Introduction}
The one-dimensional Laguerre polynomial $L_n^{\alpha}(x)$ of type $\alpha>-1$ are defined by the generating function identity
\begin{equation}\label{Generate}
\sum\limits_{n=0}^{\infty}L_n^{\alpha}(x)t^{n}=(1-t)^{-\alpha-1}e^{-\frac{xt}{1-t}},\,\,\,\arrowvert t\arrowvert<1.
\end{equation}
Here $x>0$, $n\in \mathbb{N}$. Each $L_n^{\alpha}(x)$ is a polynomial of degree $n$ explicitly given by 
$$L_n^{\alpha}(x)=\sum\limits_{j=0}^{n}\frac{\Gamma(n+\alpha+1)}{\Gamma(n-j+1)\Gamma(j+\alpha+1)}\frac{(-x)^j}{j!}.$$ 
Let us define the Laguerre functions $\varphi_n^{\alpha}(x)$ by
\begin{equation}\label{L1}
\varphi_n^{\alpha}(x)=\left(\frac{2n!}{\Gamma(n+\alpha+1)}\right)^{1/2}x^{\alpha+1/2}L_n^{\alpha}(x^2)e^{-x^2/2},
\end{equation}
which form an orthonormal basis for $L^2({\mathbb{R}_{+}})$. Now for each multi-index $\mu=(\mu_1,...,\mu_d)\in \mathbb{N}^{d}$ and $\alpha=(\alpha_1,...,\alpha_d)\in (-1,\infty)^d$, the multiple Laguerre functions are defined by the formula
\begin{equation}\label{Ld}
\varphi_{\mu}^{\alpha}(x)=\prod\limits_{j=1}^{d}\varphi_{\mu_j}^{\alpha_j}(x_j), \quad x\in \mathbb{R}_{+}^{d}.
\end{equation}
The set $\{\varphi_{\mu}^{\alpha}\}$ forms an orthonormal basis of $L^2(\mathbb{R}_{+}^d)$, precisely, for every $f\in L^2(\mathbb{R}_{+}^d)$, 
 \begin{equation}\label{exp1.1}
f(x)=\sum_{\mu\in\mathbb{N}^d}\langle f, \varphi_{\mu}^{\alpha}\rangle\varphi_{\mu}^{\alpha}(x).
\end{equation} 
\par These functions $\{\varphi_{\mu}^{\alpha}\}_{\mu\in \mathbb{N}^d}$ are the eigenfunctions of the following differential equation 
\begin{equation*}
L^{\alpha}\varphi_{\mu}^{\alpha}(x)=\left(-\Delta+\arrowvert x\arrowvert^2+\sum_{i=1}^{d}(\alpha_j^2-\frac{1}{4})\frac{1}{x_i^2}\right)\varphi_{\mu}^{\alpha}(x)=(4\arrowvert \mu\arrowvert_1+2\arrowvert \alpha\arrowvert_1+2d)\varphi_{\mu}^{\alpha}(x);
\end{equation*}
by $\arrowvert\alpha\arrowvert_1$ and $\arrowvert \mu\arrowvert_1$ we denote $\arrowvert\alpha\arrowvert_1=\alpha_1+\cdot\cdot\cdot\alpha_d$ and $\arrowvert \mu\arrowvert_1=\mu_1+\cdot\cdot\cdot \mu_d$. The operator $L^{\alpha}$ is positive and symmetric in $L^2(\mathbb{R}_{+}^d)$. The operator
$$\mathcal{L}^{\alpha} f=\sum_{\mu\in \mathbb{N}^d}(4\arrowvert \mu\arrowvert_1+2\arrowvert \alpha\arrowvert_1+2d)\langle f,\varphi_{\mu}^{\alpha}\rangle\varphi_{\mu}^{\alpha}$$
defined on the domain 
\begin{equation*}
Dom\,(\mathcal{L}^{\alpha})=\left\{f\in L^2(\mathbb{R}_{+}^d): \sum_{\mu\in \mathbb{N}^d}\arrowvert(4\arrowvert \mu\arrowvert_1+2\arrowvert \alpha\arrowvert_1+2d)\langle f,\varphi_{\mu}^{\alpha}\rangle\arrowvert^2<\infty\right\}
\end{equation*}
is a self-adjoint extension of $L^{\alpha}$, has the discrete spectrum $\{4n+2\arrowvert\alpha\arrowvert_1+2d: n\in\mathbb{N}\}$ and admits the spectral decomposition
\begin{equation*}
\mathcal{L}^{\alpha} f=\sum_{n=0}^{\infty}(4n+2\arrowvert \alpha\arrowvert_1+2d)\mathcal{P}_{n}^{\alpha}f, \quad f\in Dom\,(\mathcal{L}^{\alpha}),
\end{equation*}
where the spectral projections are 
\begin{equation*}
\mathcal{P}_{n}^{\alpha}f=\sum_{\arrowvert \mu\arrowvert_1=n}\langle f,\varphi_{\mu}^{\alpha}\rangle\varphi_{\mu}^{\alpha}.
\end{equation*}
The multi-dimensional expansions with respect to the system of Laguerre function $\{\varphi_{\mu}^{\alpha}\}_{\mu\in \mathbb{N}^d}$ are known in the literature as Laguerre function expansions of Hermite type, the terminology being motivated by the fact that when $\alpha$ is multi-index such that $\alpha_j=\pm1/2,\,j=1,\cdot\cdot\cdot,d$, then $\{\varphi_{\mu}^{\alpha}\}_{\mu\in \mathbb{N}^d}$ coincides with a multi-dimensional Hermite function. 
\par The purpose of this paper is to investigate the Bochner-Riesz means for such Laguerre expansions. This framework was initially introduced and studied by Markett in \cite{M}, and later further investigated by Thangavelu in \cite{TS3,TS4} and Stempak in \cite{SK}. Subsequently, other researchers have explored additional aspects of these Laguerre expansions within harmonic analysis, as seen in \cite{BFR,NS,NS1}.
\par The Bochner-Riesz means for such Laguerre expansions can be defined through the spectral decomposition of the Hermite-Laguerre operator $\mathcal{L}^{\alpha}$. Specifically, for $\lambda \geq 0$ and $R > 0$, the Bochner-Riesz means $S_{R}^{\lambda}(\mathcal{L}^{\alpha})f(x)$ of order $\lambda$ on $\mathbb{R}_+^{d}$ are defined as:
\begin{equation}\label{BR}
S_{R}^{\lambda}(\mathcal{L}^{\alpha})f(x)=\sum_{n=0}^{\infty}\left(1-\frac{e_n}{R^2}\right)_{+}^{\lambda}\mathcal{P}_n^{\alpha}f(x),
\end{equation}
where $t_+=\max\{0, t\}$ for  $t\in \mathbb{R}$ and $e_n=4n+2\arrowvert\alpha\arrowvert_1+2d$ denotes the n-th eigenvalue of $\mathcal{L}^{\alpha}$. 
\par The problem known as the Bochner-Riesz problem is to determine the optimal index \(\lambda\) for \(1 \leq p \leq \infty\) such that \(S_{R}^{\lambda}(\mathcal{L}^{\alpha})f\) converges to \(f\) in \(L^p(\mathbb{R}_{+}^d)\), or pointwise almost everywhere, for every \(f \in L^p(\mathbb{R}_{+}^d)\). This problem was originally formulated for the Laplacian \(-\Delta = -\sum_{i=1}^{d} \partial_{x_i}^2\) on Euclidean space \(\mathbb{R}^d\), with \(d \geq 2\). The convergence of Bochner-Riesz means \(S_{R}^{\lambda}(-\Delta)f\) in \(L^p\)-spaces is one of the most fundamental and longstanding problems in classical harmonic analysis. The central conjecture, known as the \emph{Bochner-Riesz conjecture}, asserts that for \(1 \leq p \leq \infty\), \(p \neq 2\), the Bochner-Riesz means \(S_{R}^{\lambda}(-\Delta)f\) converge in \(L^p(\mathbb{R}^d)\) if and only if
\begin{equation}\label{Section1.2}
\lambda > \lambda(p) := \max\left\{d\left|\frac{1}{2} - \frac{1}{p} \right| - \frac{1}{2}, \, 0\right\}.
\end{equation}
Herz \cite{Herz} showed that the condition \eqref{Section1.2} on \(\lambda\) is necessary for convergence in \(L^p\). Carleson and Sjölin \cite{Carleson} proved the conjecture for \(d = 2\). Significant progress has since been made in higher dimensions (see, e.g., \cite{ Bourgain, Guth, Lee, Tao2, Tao3} and references therein), although the conjecture remains open for \(d \geq 3\).

Regarding pointwise convergence, Carbery, Rubio de Francia, and Vega \cite{Carbery1} established almost everywhere convergence with the sharp summability exponent for all \(f \in L^p(\mathbb{R}^d)\), namely:
\begin{equation}\label{Section1.3}
\lim_{R \rightarrow \infty} S_{R}^{\lambda}(-\Delta)f = f \quad \text{a.e.}
\end{equation}
provided \(p \geq 2\) and \(\lambda > \lambda(p)\). For \(d = 2\), this result was earlier obtained by Carbery \cite{Carbery}, who proved the sharp \(L^p\) estimates for the maximal Bochner-Riesz means. See also Christ \cite{Christ} for earlier partial results in higher dimensions, based on maximal Bochner-Riesz estimates. For the most recent developments, we refer the reader to \cite{Lee1}.

It is noteworthy that the result by Carbery et al. \cite{Carbery1} resolved the almost everywhere convergence problem up to the sharp index \(\lambda(p)\) for \(2 \leq p \leq \infty\). There are also results concerning convergence at the critical index \(\lambda = \lambda(p)\) (see, e.g., \cite{Annoni, Lee2}). However, the case \(1 < p < 2\) exhibits a markedly different nature, and few results are known except in dimension \(d = 2\); see \cite{Li, Tao, Tao1} for notable contributions in this range.
\par The Bochner-Riesz means for elliptic operators (e.g. Schr\"{o}dinger operator with potentials, the Laplacian on manifolds) have attracted a lot of attention and have been studied extensively by many authors. See, for example, \cite{COSY, CLSL, CS, DOS, GHS, JLR, JR, KST, LR, MYZ, S} and references therein. 
\par Returning to the Bochner-Riesz means \( S_{R}^{\lambda}(\mathcal{L}^{\alpha}) \) associated with the Hermite-Laguerre operator \(\mathcal{L}^{\alpha}\). As pointed out in \cite{TS3}, the \(L^p\)-convergence of \(S_{R}^{\lambda}(\mathcal{L}^{\alpha})f\) for \(\alpha \in [-\tfrac{1}{2}, \infty)^d\) and \(1 < p < \infty\) can be deduced directly from the corresponding results for multiple Hermite expansions via a transplantation theorem (see, for instance, Theorem 6.4.2 in \cite{TS3}). Moreover, since Thangavelu \cite{TS2} demonstrated that the condition \(\lambda > \lambda(p)\) is necessary for the \(L^p\)-convergence of the Bochner-Riesz means associated with multiple Hermite expansions, it follows, by the transplantation result, that the same condition \eqref{Section1.2} on \(\lambda\) is also necessary for convergence of \(S_{R}^{\lambda}(\mathcal{L}^{\alpha})f\) in \(L^p(\mathbb{R}_+^d)\). For a more precise account of the known ranges of \(\lambda\) and \(1 < p < \infty\) in which \(S_{R}^{\lambda}(\mathcal{L}^{\alpha})f\) converges (or fails to converge) to \(f\) in \(L^p(\mathbb{R}_{+}^d)\) for every \(f \in L^p(\mathbb{R}_{+}^d)\), we refer the reader to \cite{KGB, TS, TS1}, as well as to the recent developments in \cite[Proposition 4.1]{LR}.
\par Motivated by the recent works \cite{CD, JLR} on a.e. convergence of Bochner-Riesz means for the Hermite operator $H:=-\Delta +\arrowvert x\arrowvert^2$ and the twisted Laplacian on $\mathbb{C}^{d}$ given by 
\begin{equation*}
\mathcal{L}=-\sum_{j=1}^{d}\left(\left(\frac{\partial}{\partial x_j}-\frac{i}{2}y_j\right)^2+\left(\frac{\partial}{\partial y_j}+\frac{i}{2}x_j\right)^2\right),\quad (x,y)\in \mathbb{R}^{d}\times\mathbb{R}^{d}.
\end{equation*}
This paper is concerned with the almost everywhere convergence of the spectral multipliers \( S_{R}^{\lambda}(\mathcal{L}^{\alpha})f \). Specifically, we aim to characterize the values of \(\lambda\) for which
\begin{equation}\label{Almost}
\lim_{R \to \infty} S_{R}^{\lambda}(\mathcal{L}^{\alpha})f(x) = f(x) \quad \text{a.e.} \quad \forall f \in L^p(\mathbb{R}_{+}^{d}),
\end{equation}
for \( 2 \leq p < \infty \). Compared to the \( L^p \)-norm convergence of \( S_{R}^{\lambda}(\mathcal{L}^{\alpha})f \), its pointwise convergence has received relatively little attention, and the existing literature remains limited.

In the one-dimensional case, Thangavelu \cite{TS1} showed that when \( \alpha \geq \tfrac{1}{2} \), the convergence in \eqref{Almost} holds for \( 1 < p < \infty \) provided that \( \lambda > \tfrac{1}{6} \), which is strictly greater than the critical index \( \lambda(p) \). Later, Stempak \cite{SK} extended the range to \( \alpha \geq 0 \), but with a more restrictive condition on \(\lambda\); he proved that \eqref{Almost} holds for \( 1 \leq p < \infty \) if \( \lambda > \alpha + \tfrac{2}{3} \). In high-dimensional settings, we have not found any relevant research on this topic.

In this paper, we go beyond the one-dimensional setting and consider the higher-dimensional case as well. Moreover, we study the full range of parameters \(\alpha \in [-1/2, \infty)^d\), which is natural for two reasons. First, when each \(\alpha_j = \pm 1/2\) for \( j = 1, \dots, d \), the operator \(\mathcal{L}^\alpha\) reduces to the \(d\)-dimensional harmonic oscillator \( H := -\Delta + |x|^2 \). Second, it is only for \(\alpha \in [-1/2, \infty)^d\) that the eigenfunctions \(\varphi_{\mu}^{\alpha}\) belong to all \(L^p(\mathbb{R}_+^d)\) spaces for \(1 \leq p < \infty\) (see \cite{NS}).
\par The following is the main result of this paper, which almost completely settles the a.e. convergence problem of $S_{R}^{\lambda}(\mathcal{L}^{\alpha})f$ except the endpoint cases.
\begin{theorem}\label{T1.1}
Let $\alpha \in[-1/2,\infty)^d$, $2\leq p<\infty$, and $\lambda\geq0$. Then, for any $f\in L^p(\mathbb{R}_{+}^{d})$, we have \eqref{Almost} whenever $\lambda>\lambda(p)/2$. In particular, for $d=1$, \eqref{Almost} holds for all $f\in L^p(\mathbb{R}_{+}^{d})$ whenever $\lambda>0$. Conversely, \eqref{Almost} fails for $\lambda<\lambda(p)/2$ when $2d/(d-1)<p<\infty$.
\end{theorem}
From Theorem \ref{T1.1}, we observe that the critical summability index for almost everywhere convergence is half of that required for $L^p$ convergence. A similar result was established in \cite{CD,JLR} for the Bochner-Riesz means associated with the Hermite operator and the twisted Laplacian. This improvement in the summability index stems from the fact that these operators possess discrete spectra bounded away from zero. This spectral property enables the inhomogeneous weight $(1+\arrowvert x\arrowvert )^{-\beta}$, employed in Theorem \ref{T1.2} below, to be effective. The summability indices in Theorem \ref{T1.1} align with those studied in \cite{CD} when \(\alpha_j = \pm 1/2\) for \( j = 1, \dots, d \). Theorem~\ref{T1.1} further shows that the critical summability index $\lambda(p)/2$ remains invariant under inverse-square potential perturbations. This contrasts with results for the Hermite operator $H := -\Delta + |x|^2$ in \cite{CD}, where the critical index for a.e. convergence of $S_{R}^{\lambda}(H)f$ is half that of the classical Bochner-Riesz means $S_{R}^{\lambda}(-\Delta)f$. This phenomenon implies that the influence of the Hermite potential $\arrowvert x\arrowvert^2$ on the Laplacian $-\Delta$ is stronger in magnitude than that of the inverse-square potential term $\sum_{i=1}^{d}\left(\alpha_j^2-\frac{1}{4}\right)\frac{1}{x_i^2}$, or in other words, the influence of the inverse-square potential $\sum_{i=1}^{d}\left(\alpha_j^2-\frac{1}{4}\right)\frac{1}{x_i^2}$ can be viewed as a small perturbation of the harmonic oscillator, in the sense that it preserves the discreteness of the spectrum.
\par The sufficiency part of Theorem \ref{T1.1} relies on the maximal estimate, a standard tool in the study of almost everywhere convergence. To prove \eqref{Almost}, we introduce the corresponding maximal operator \( S_{*}^{\lambda}(\mathcal{L}^{\alpha}) \), defined by
\[
S_{*}^{\lambda}(\mathcal{L}^{\alpha})f(x) := \sup_{R>0} \left| S_{R}^{\lambda}(\mathcal{L}^{\alpha})f(x) \right|.
\]
We will establish the following weighted estimate, from which the almost everywhere convergence of \( S_{R}^{\lambda}(\mathcal{L}^{\alpha})f \) follows through standard arguments.
\begin{theorem}\label{T1.2}
Let $\alpha \in[-1/2,\infty)^d$ and $0\leq\beta<d$. The maximal operator $S_{*}^{\lambda}(\mathcal{L}^{\alpha})$ is bounded on $L^2(\mathbb{R}_{+}^{d},(1+\arrowvert x\arrowvert )^{-\beta})$\footnote{$\lVert f\rVert_{L^2(\mathbb{R}_{+}^{d},(1+\arrowvert x\arrowvert )^{-\beta})}^2:=\int_{\mathbb{R}_{+}^{d}}\arrowvert f(x)\arrowvert^2(1+\arrowvert x\arrowvert )^{-\beta}dx$} if 
\begin{equation}
\lambda>\max\{\frac{\beta-1}{4},0\}.
\end{equation}
\end{theorem}
Via a standard approximation argument (see, for example, \cite{SW} and \cite[Theorem 2]{TS5}) Theorem \ref{T1.2} establishes a.e. convergence of \( S_{R}^{\lambda}(\mathcal{L}^{\alpha})f \) for all $f\in L^2(\mathbb{R}_{+}^{d},(1+\arrowvert x\arrowvert )^{-\beta})$ provided that $\lambda>\max\{\frac{\beta-1}{4},0\}$. Then, the sufficiency part of Theorem \ref{T1.1} follows directly from the embedding  
\[ L^p(\mathbb{R}_{+}^d) \hookrightarrow L^2(\mathbb{R}_{+}^{d}, (1 + |x|)^{-\beta}) \quad \text{when} \quad \beta > d(1 - 2/p) \ \text{and} \ p \geq 2. \]

\emph{Our approach.} The methods used in this paper are inspired by \cite{CD, JLR}. For the proof of Theorem \ref{T1.2}, as in \cite{CD}, we begin by establishing a weighted estimate (Lemma \ref{weighted}), analogous to \cite[Lemma 1.4]{CD}, which is closely related to the fact that the discrete spectrum of $\mathcal{L}^{\alpha}$ is bounded away from zero. However, because the operator $\mathcal{L}^{\alpha}$ contains an inverse-square potential term, the strategy proposed in \cite{CD} is constrained by Hardy's inequality to the parameter range $\alpha_j \geq \tfrac{3}{2}$ for all $j = 1, \dots, d$, and thus cannot be extended to the full domain $\alpha \in [-1/2, \infty)^d$.

For the twisted Laplacian $\mathcal{L}$, an analogous result to \cite[Lemma 1.4]{CD} cannot be established using the method employed in \cite{CD}. Instead, \cite{JLR} introduces an alternative kernel estimation method. Given a function $\eta \in C_c^{\infty}((-2,2))$, their approach relies on estimates for the kernel of the multiplier operator $\eta((\mu - \mathcal{L})/R)$, which in turn is expressed via $e^{it\mathcal{L}}(z,z')$, the kernel of the Schrödinger propagator $e^{it\mathcal{L}}$. This method involves oscillatory integral estimates, where integration by parts becomes essential. Since the kernel of the Schrödinger propagator $e^{it\mathcal{L}^{\alpha}}$ is given by (for the one-dimensional case, see \cite{TS4}):
\begin{equation}\label{SK}
e^{it\mathcal{L}^{\alpha}}(x,y)=(\sin 2t)^{-d}\exp\left(-\frac{i}{2}\cot(2t)(|x|^2+|y|^2)\right)\prod_{i=1}^{d}\sqrt{x_iy_i}J_{\alpha_i}\left(\frac{x_iy_i}{\sin 2t}\right),
\end{equation}
where $J_{\alpha_i}$ denotes Bessel functions of order $\alpha_i$ for $i=1,...,d$. As demonstrated in \cite{TS4}, this representation presents two significant challenges: first, the analytical treatment of this expression is inherently complex due to the interplay of oscillatory and Bessel function components; second, the application of integration by parts techniques - crucial for establishing key estimates - necessarily imposes the parameter constraint $\alpha_j \geq \tfrac{1}{2}$ for all $j=1,...,d$.

To address the limitations of both aforementioned approaches, we instead derive kernel estimates for the potential operators $(1 + \mathcal{L}^{\alpha})^{-\beta}$ with $\beta > 0$. These kernels are represented via the heat-diffusion semi-group $e^{-t\mathcal{L}^{\alpha}}$, allowing us to avoid oscillatory integrals altogether, which makes the analysis simpler. This strategy also yields the corresponding estimates valid for all  parameters $\alpha \in [-1/2, \infty)^d$.

Regarding the necessity part of Theorem \ref{T1.1}, since $S_{R}^{\lambda}(\mathcal{L}^{\alpha})f(x)$ does not preserve radiality if $f$ is radial. As a result, the method in \cite{CD} is no longer applicable, where the authors showed the sharpness of summability index by making use of the Nikishin-Maurey theorem. Motivated by \cite{JLR}, where the authors constructed counterexamples by using the function $e^{it\mathcal{L}}(z,0)$. However, clearly by \eqref{SK}, $e^{it\mathcal{L}^{\alpha}}(x,0)$ vanishes identically whenever $\alpha\neq (-\frac{1}{2},\cdot\cdot\cdot,-\frac{1}{2})$, which amounts us some additional work. Based on the same fashion of deriving the kernel of the Schrödinger propagator $e^{it\mathcal{L}^{\alpha}}$, we proceed to construct counterexamples using the generating function formula \eqref{Generate} to demonstrate that almost everywhere convergence fails when $\delta < \lambda(p)/2$. More precisely, for $2d/(d - 1) < p \leq \infty$, we show that there exists a function $f \in L^p(\mathbb{R}_+^d)$ such that  
\begin{equation}\label{Fail}
|\{x \in \mathbb{R}_+^d : \sup_{R > 0} |S_{R}^{\lambda}(\mathcal{L}^{\alpha})f(x)| = \infty \}| \gtrsim 1
\end{equation}
whenever $\lambda < \lambda(p)/2$.
\subsection{Organization}
In Section \ref{Section2}, we establish several preparatory estimates that are crucial for the proof of Theorem \ref{T1.2}. Section \ref{Section3} is dedicated to proving Theorem \ref{T1.2}. In Section \ref{Section4}, we demonstrate the sharpness of the summability indices, thereby completing the proofs of Theorem \ref{T1.1}.
\subsection{Notation}
For a given positive numbers $A$, $B$, $A\lesssim B$ means $A\leq CB$ for a constant $C>0$ depending only on $d$ and \( \alpha \). We write $A\sim B$ if $A\lesssim B$ and $B\lesssim A$. If the constant $C$ is allowed to depend on an additional parameter such as $p$, then we denote this by $A\lesssim_{p} B$.
  We emphasize that the notations $\arrowvert x\arrowvert_1$ and $\arrowvert x\arrowvert$ are distinct, where $\arrowvert x\arrowvert_1:=\sum_{i=1}^{d}x_i$ and $\arrowvert x\arrowvert:=\left(\sum_{i=1}^{d}x_i^2\right)^{1/2}$. For a subset $E\subset\mathbb{R}_{+}^{d}$, the notation $\arrowvert E\arrowvert$ denotes the Lebesgue measure of $E$. For any bounded function $\mathscr{M}$ the operator $\mathscr{M}(\mathcal{L}^{\alpha})$ is defined by $\mathscr{M}(\mathcal{L}^{\alpha}):=\sum_{n=0}^{\infty}\mathscr{M}(4n+2\arrowvert \alpha\arrowvert_1+2d)\mathcal{P}_n$.
\section{Pleliminaries}\label{Section2}
In this section we obtain some estimates as in \cite{CD}, which we use to prove the main results. Before proceeding, let us invoke the heat-diffusion semi-group $\{e^{-t\mathcal{L}^{\alpha}}\}_{t\geq0}$, defined by means of the spectral theorem,
\begin{equation*}
e^{-t\mathcal{L}^{\alpha}}f=\sum_{n=0}^{\infty}e^{-t(4n+2\arrowvert\alpha\arrowvert_1+2d)}\mathcal{P}_n^{\alpha}f,\quad f\in L^2(\mathbb{R}_{+}^d),
\end{equation*}
is a strongly continuous semi-group of contractions on $L^2(\mathbb{R}_{+}^d)$ for which $\mathcal{L}^{\alpha}$ is the infinitesimal generator. We have the following integral representation of $\{e^{-t\mathcal{L}^{\alpha}}\}$:
\begin{equation*}
e^{-t\mathcal{L}^{\alpha}}f=\int_{\mathbb{R}_{+}^d}K_t^{\alpha}(x,y)f(y)dy,\quad f\in L^2(\mathbb{R}_{+}^d),\quad f\in L^2(\mathbb{R}_{+}^d),\,x\in\mathbb{R}_{+}^d, 
\end{equation*}
where
\begin{equation*}
K_t^{\alpha}(x,y)=\sum_{n=0}^{\infty}e^{-t(4n+2\arrowvert\alpha\arrowvert_1+2d)}\sum_{\arrowvert \mu\arrowvert_1=n}\varphi_\mu^{\alpha}(x)\varphi_\mu^{\alpha}(y),\,\,x,\,y\,\in\mathbb{R}_{+}^d.
\end{equation*}
It is known, cf. \cite[2.3]{NS}, that
\begin{equation}\label{Kernel}
K_t^{\alpha}(x,y)=(\sinh 2t)^{-d}\exp\left(-\frac{1}{2}\coth(2t)(\arrowvert x\arrowvert^2+\arrowvert y\arrowvert^2)\right)\prod_{i=1}^{d}\sqrt{x_iy_i}I_{\alpha_i}\left(\frac{2x_iy_i}{\sinh 2t}\right).
\end{equation}
Here $I_{\nu}(z)=i^{-\alpha}J_{\alpha}(iz)$ is the modified Bessel function of order $\nu$ ($J_{\nu}$  being the standand Bessel function of the first kind). 
\par Let $W_s(x)=(4\pi s)^{-d/2}\exp(-\arrowvert x\arrowvert^2/(4s))$ be the usual Gauss-Weierstrass kernel in $\mathbb{R}^d$. Then we have 
\begin{proposition}\cite[Proposition 2.1]{NS}. Given $\alpha\in [-1/2,\infty)^d$, there exists a constant $C_\alpha$ such that
\begin{equation}
K_t^{\alpha}(x,y)\leq C_{\alpha}W_{t}(x-y),\quad t>0,\,\,x,\,y\,\in\mathbb{R}_{+}^d.
\end{equation}
\end{proposition}
The Gaussian upper bounds imply the following finite-speed propagation property of the wave operator $\cos t\sqrt{\mathcal{L}^{\alpha}}$.
\begin{lemma}\label{Finit}
There exists a constant $c_0>0$ such that
\begin{equation}\label{Finit.1}
supp\,K_{\cos(t\sqrt{\mathcal{L}^{\alpha}})}(x,y)\subset D(t):=\{(x,y)\in \mathbb{R}_{+}^{d}\times\mathbb{R}_{+}^{d}:\arrowvert x-y\arrowvert\leq c_0t\},\quad \forall t>0,
\end{equation}
where $K_{\cos(t\sqrt{\mathcal{L}^{\alpha}})}(x,y)$ is the kernel associated to $\cos(t\sqrt{\mathcal{L}^{\alpha}})$.
\end{lemma}
For the proof of Lemma \ref{Finit}, we refer to \cite{HLMML,SA} for example. We also need the following Littlewood-Paley inequality for the Laguerre operator $\mathcal{L}^{\alpha}$.
\begin{proposition}\label{PL}
Fix a non-zero $C^{\infty}$ bump function $\psi$ on $\mathbb{R}$ such that supp\,$\psi\subseteq(\frac{1}{2},2)$. Let $\psi_k(t)=\psi(2^{-k}t)$, $k\in\mathbb{Z}$, for $t>0$. Then, for any $-d<\beta<d$,
\begin{equation}\label{PL1.1}
\left\lVert \left(\sum_{k=-\infty}^{\infty}\arrowvert \psi_k(\sqrt{\mathcal{L}^{\alpha}})f\arrowvert^2\right)^{1/2}\right\rVert_{L^2(\mathbb{R}_{+}^d,(1+\arrowvert x\arrowvert)^{\beta})}\leq C\lVert  f\rVert_{L^2(\mathbb{R}_{+}^d,(1+\arrowvert x\arrowvert)^{\beta})}.
\end{equation}
\end{proposition}
\begin{proof} The argument is somewhat standard, see, for example, \cite[41, Chapter]{Stein}. Moreover, the proof provided in \cite{CD} applies to general differential operators whose heat kernels satisfy Gaussian upper bounds, we omit the details here and refer readers to \cite[p1.6]{CD} for a complete treatment. It should be emphasized that its proof requires the weight function $(1+|x|)^{\beta}$ to belong to the $A_2$ class, and $(1+|x|)^{\beta}$ belongs to the $A_2$ class if and only if  $-d<\beta<d$, which is precisely why the condition  $-d<\beta<d$ is needed. For a good introduction of $A_p$ theory, we refer readers to \cite{GL, MS}.
\end{proof}
\subsection{Weighted estimates of the negative power $(1+\mathcal{L}^{\alpha})^{-\beta}$}
The ideas of this section is inspired by \cite{BT, NS1}. Given $\beta\geq0$, we proceed to prove a weighted estimates for the negative power $(1+\mathcal{L}^{\alpha})^{-\beta}$. In view of the spectral theorem, it is expressed on $L^2(\mathbb{R}_{+}^d)$ by the spectral series 
\begin{equation}\label{R3.1}
(1+\mathcal{L}^{\alpha})^{-\beta}f(x)=\sum_{\mu\in \mathbb{N}^d}(4\arrowvert\mu\arrowvert_1+2\arrowvert\alpha\arrowvert_1+2)^{-\beta}\langle f,\varphi_{\mu}^{\alpha}(x)\rangle\varphi_{\mu}^{\alpha}(x)\quad x\in \mathbb{R}_{+}^d,\quad f\in L^2(\mathbb{R}_+^d).
\end{equation}
The main goal in this subsection is to establish the following weighted estimate for the operators $(1+\mathcal{L}^{\alpha})^{-\beta}$, which will be crucial for proving the low-frequency part of the square function estimate (Proposition \ref{SP} below).
\begin{lemma}\label{weighted}
Let $\beta\geq 0$. Then, the estimate 
\begin{equation}\label{weighted1.1}
\lVert (1+\arrowvert x\arrowvert)^{2\beta}(1+\mathcal{L}^{\alpha})^{-\beta}f\rVert_{L^2(\mathbb{R}_{+}^{d})}\leq C\lVert f\rVert_{L^2(\mathbb{R}_{+}^{d})}
\end{equation}
holds for any  $f\in L^2(\mathbb{R}_{+}^{d})$.
\end{lemma}
It is clear that \eqref{weighted1.1} holds for $\beta=0$. For every $\beta>0$, to prove \eqref{weighted1.1}, following \cite{BT, NS1}, we introduce the following operator. 
\begin{definition}
Given $\beta>0$, we define for $f\in \mathbb{S}_{\mathcal{L}^{\alpha}}$, the operator 
\begin{equation}
T_{\alpha}^{\beta}f(x)=\frac{1}{\Gamma(\beta)}\int_{0}^{\infty}e^{-t(1+\mathcal{L}^{\alpha})}f(x)t^{\beta}\frac{dt}{t},\quad\quad x\in\mathbb{R}_{+}^d,
\end{equation}
where $\{e^{-t\mathcal{L}^{\alpha}}\}_{t\geq0}$ is the heat semi-group associated to $\mathcal{L}^{\alpha}$ and $\mathbb{S}_{\mathcal{L}^{\alpha}}$ denotes the set of finite linear combinations of eigenfunctions of the operator $\mathcal{L}^{\alpha}$.
\end{definition}
\begin{remark}
If $\beta>0$ and $\mu\in \mathbb{N}^d$, by using the $\Gamma$ function and the fact
$$ e^{-t(1+\mathcal{L}^{\alpha})}\varphi_{\mu}^{\alpha}=e^{-t(4\arrowvert\mu\arrowvert_1+2\arrowvert\alpha\arrowvert_1+2)}\varphi_{\mu}^{\alpha},$$
we have 
\begin{equation*}
T_{\alpha}^{\beta}\varphi_{\mu}^{\alpha}(x)=\frac{1}{\Gamma(\beta)}\int_{0}^{\infty}e^{-t(1+\mathcal{L}^{\alpha})}\varphi_{k}^{\alpha}(x)t^{\beta}\frac{dt}{t}=(4\arrowvert\mu\arrowvert_1+2\arrowvert\alpha\arrowvert_1+2)^{-\beta}\varphi_{\mu}^{\alpha}(x),\quad x\in \mathbb{R}_{+}^d.
\end{equation*}
We will see later that this definition will consistent with \eqref{R3.1} (see also \cite{NS1}) for every $f\in L^2(\mathbb{R}_+^d)$.
\end{remark}
\begin{proposition}\label{FK}
For every $\alpha\in [-1/2,\infty)^d$, if $\beta>0$, then the operator $T_{\alpha}^{\beta}$ has integral representation
\begin{equation}
T_{\alpha}^{\beta}f(x)=\int_{\mathbb{R}_{+}^{d}}K^{\alpha}(\beta;x,y)f(y)dy, \quad\quad x\in\mathbb{R}_{+}^{d},
\end{equation}
for all  $f\in \mathbb{S}_{\mathcal{L}^{\alpha}}$. There exists a constant $C$ depends only on $\alpha$, $\beta$ and $d$ such that 
\begin{equation}\label{C1.1}
K^{\alpha}(\beta;x,y)\leq C\Phi^{\alpha}(\beta;x-y),\quad\quad for\quad all\quad x,y\in\mathbb{R}_{+}^d.
\end{equation}
where $\Phi^{\alpha}(\beta;x)$ is defined by
\begin{equation}\label{C1.2}
\Phi^{\alpha}(\beta;x)=\left\{\begin{array}{lr}
\frac{\mathcal{X}_{\arrowvert x\arrowvert<1}(x)}{\arrowvert x\arrowvert^{d-2\beta}}+e^{-\frac{1}{4}\arrowvert x\arrowvert^2}\mathcal{X}_{\arrowvert x\arrowvert\geq 1}(x),\quad\quad\quad\quad\quad\qquad if\,\,\beta<\frac{d}{2},\\
\log\left(\frac{e}{\arrowvert x\arrowvert}\right)\mathcal{X}_{\arrowvert x\arrowvert<1}(x)+e^{-\frac{1}{4}\arrowvert x\arrowvert^2}\mathcal{X}_{\arrowvert x\arrowvert\geq 1}(x),\quad\quad\quad if\,\,\beta=\frac{d}{2},\\
\mathcal{X}_{\arrowvert x\arrowvert<1}(x)+e^{-\frac{1}{4}\arrowvert x\arrowvert^2}\mathcal{X}_{\arrowvert x\arrowvert\geq 1}(x),\quad\quad\quad\quad\quad\quad\,\,\,if\,\,\beta>\frac{d}{2}.
\end{array}
\right.
\end{equation}
\end{proposition}
\begin{proof}
Our proof here is inspired by the argument in \cite[Proposition 2]{BT} where the authors discussed the Sobolev spaces associated to the harmonic oscillator. If $f\in \mathbb{S}_{\mathcal{L}^{\alpha}}$, then for $x\in \mathbb{R}_{+}^d$,
\begin{equation}
\begin{aligned}
(1+\mathcal{L}^{\alpha})^{-\beta}f(x)&=\frac{1}{\Gamma(\beta)}\int_{0}^{\infty}e^{-t(1+\mathcal{L}^{\alpha})}f(x)t^{\beta}\frac{dt}{t}\\
&=\frac{1}{\Gamma(\beta)}\int_{0}^{\infty}\int_{\mathbb{R}_{+}^d}K_{t}^{\alpha}(x,y)f(y)dye^{-t}t^{\beta}\frac{dt}{t},
\end{aligned}
\end{equation}
where $K_{t}^{\alpha}(x,y)$ is given by \eqref{Kernel}.
Therefore, if we show that for some constant $C$,
\begin{equation}
\frac{1}{\Gamma(\beta)}\int_{0}^{\infty}K_{t}^{\alpha}(x,y)e^{-t}t^{\beta}\frac{dt}{t}\leq C\Phi^{\alpha}(\beta;x-y),\quad\quad for\quad all\quad x,y\in\mathbb{R}_{+}^d,
\end{equation}
where $\Phi^{\alpha}\in L^{1}(\mathbb{R}^{d})$, then by Fubini's theorem,
\begin{equation}
(1+\mathcal{L}^{\alpha})^{-\beta}f(x)=\int_{\mathbb{R}_{+}^{d}}K^{\alpha}(\beta;x,y)f(y)dy, \quad\quad x\in\mathbb{R}_{+}^d,
\end{equation}
with
\begin{equation}
K^{\alpha}(\beta;x,y)=\frac{1}{\Gamma(\beta)}\int_{0}^{\infty}K_{t}^{\alpha}(x,y)e^{-t}t^{\beta}\frac{dt}{t}.
\end{equation}
We perform the change of variables  $t=\frac{1}{2}\log\left(\frac{1+s}{1-s}\right)$ and by \eqref{Kernel}, we obtain
\begin{equation}\label{KP}
K^{\alpha}(\beta;x,y)=\frac{1}{\Gamma(\beta)2^{\frac{d}{2}+\beta-1}}\int_{0}^{1}\zeta_{\beta}(s)e^{-\frac{1}{4}(s+\frac{1}{s})(\arrowvert x\arrowvert^2+\arrowvert y\arrowvert^2)}\prod_{i=1}^{d}\sqrt{\frac{1-s^2}{2s}x_iy_i}I_{\alpha_i}\left(\frac{1-s^2}{2s}x_iy_i\right)\frac{ds}{s},
\end{equation}
where
\begin{equation}
\zeta_{\beta}(s)=\sqrt{\frac{1-s}{1+s}}\left(\frac{1-s^2}{s}\right)^{\frac{d}{2}-1}\left\{ \log\left(\frac{1+s}{1-s}\right)\right\}^{\beta-1}.
\end{equation}
To proceed, we need some estimates for the modified Bessel function $I_{\nu}(z)$, $\nu>-1$. For $z\in\mathbb{R}_{+}$, we have
\begin{equation*}
\sqrt{z}I_{\nu}(z)=z^{\nu+\frac{1}{2}},\quad z\rightarrow 0,\quad and\quad\sqrt{z}I_{\nu}(z)=\frac{1}{\sqrt{2\pi}}e^{z}(1+O(\frac{1}{z})), \quad z\rightarrow\infty, 
\end{equation*}
which implies that there exist a constant $C_\alpha$ such that for $x=(x_1,\cdot\cdot\cdot,x_d)\in \mathbb{R}_{+}^d$ and $\alpha\in [-1/2,\infty)^d$,
\begin{equation}\label{I3}
\prod_{i=1}^{d}\sqrt{x_i}I_{\alpha}(x_i)\leq C_\alpha e^{x_1+\cdot\cdot\cdot+x_d}.
\end{equation}
Now, we split $K^{\alpha}(\beta;x,y)$ as $K^{\alpha}(\beta;x,y)=K^{\alpha}_1(\beta;x,y)+K^{\alpha}_2(\beta;x,y)$, where
\begin{equation*}
K^{\alpha}_1(\beta;x,y)=\frac{1}{\Gamma(\beta)2^{\frac{d}{2}+\beta-1}}\int_{0,\frac{1}{2}]}\zeta_{\beta}(s)e^{-\frac{1}{4}(s+\frac{1}{s})(\arrowvert x\arrowvert^2+\arrowvert y\arrowvert^2)}\prod_{i=1}^{d}\sqrt{\frac{1-s^2}{2s}x_iy_i}I_{\alpha_i}\left(\frac{1-s^2}{2s}x_iy_i\right)\frac{ds}{s},
\end{equation*}
and 
\begin{equation*}
K^{\alpha}_2(\beta;x,y)=\frac{1}{\Gamma(\beta)2^{\frac{d}{2}+\beta-1}}\int_{1/2}^{1}\zeta_{\beta}(s)e^{-\frac{1}{4}(s+\frac{1}{s})(\arrowvert x\arrowvert^2+\arrowvert y\arrowvert^2)}\prod_{i=1}^{d}\sqrt{\frac{1-s^2}{2s}x_iy_i}I_{\alpha_i}\left(\frac{1-s^2}{2s}x_iy_i\right)\frac{ds}{s}.
\end{equation*}
It is easier to estimate $K^{\alpha}_2(\beta;x,y)$. By the estimate \eqref{I3}, we obtain
\begin{equation}
\begin{aligned}
K^{\alpha}_2(\beta;x,y)&=\frac{1}{\Gamma(\beta)2^{\frac{d}{2}+\beta-1}}\int_{1/2}^{1}\zeta_{\beta}(s)e^{-\frac{1}{4}(s+\frac{1}{s})(\arrowvert x\arrowvert^2+\arrowvert y\arrowvert^2)}\prod_{i=1}^{d}\sqrt{\frac{1-s^2}{2s}x_iy_i}I_{\alpha_i}\left(\frac{1-s^2}{2s}x_iy_i\right)\frac{ds}{s}\\
&\leq\frac{C_{\alpha}}{\Gamma(\beta)2^{\frac{d}{2}+\beta-1}}\int_{1/2}^{1}\zeta_{\beta}(s)e^{-\frac{1}{4}(s+\frac{1}{s})(\arrowvert x\arrowvert^2+\arrowvert y\arrowvert^2)}e^{\frac{1-s^2}{2s}x\cdot y}\frac{ds}{s}\\
&=\frac{C_{\alpha}}{\Gamma(\beta)2^{\frac{d}{2}+\beta-1}}\int_{1/2}^{1}\zeta_{\beta}(s)e^{-\frac{1}{4}s(\arrowvert x+y\arrowvert^2-\frac{1}{4s}\arrowvert x-y\arrowvert^2}\frac{ds}{s}.
\end{aligned}
\end{equation}
Since the integral
\begin{equation*}
\int_{1/2}^{1}\zeta_{\beta}(s)\frac{ds}{s}<\infty,
\end{equation*}
thus we obtain
\begin{equation}\label{P1}
K^{\alpha}_2(\beta;x,y)\leq Ce^{-\frac{1}{4}\arrowvert x-y\arrowvert^2}e^{-\frac{1}{8}\arrowvert x+y\arrowvert^2}\leq  Ce^{-\frac{1}{4}\arrowvert x-y\arrowvert^2}.
\end{equation}
Next, we estimate the term $K^{\alpha}_1(\beta;x,y)$. Again by \eqref{I3}, we obtain that
\begin{equation}\label{3.18}
\begin{aligned}
K^{\alpha}_1(\beta;x,y)&=\frac{1}{\Gamma(\beta)2^{\frac{d}{2}+\beta-1}}\int_{0}^{1/2}\zeta_{\beta}(s)e^{-\frac{1}{4}(s+\frac{1}{s})(\arrowvert x\arrowvert^2+\arrowvert y\arrowvert^2)}\prod_{i=1}^{d}\sqrt{\frac{1-s^2}{2s}x_iy_i}I_{\alpha_i}\left(\frac{1-s^2}{2s}x_iy_i\right)\frac{ds}{s}\\
&\leq\frac{C_{\alpha}}{\Gamma(\beta)2^{\frac{d}{2}+\beta-1}}\int_{0}^{1/2}\zeta_{\beta}(s)e^{-\frac{1}{4}s\arrowvert x+y\arrowvert^2-\frac{1}{4s}\arrowvert x-y\arrowvert^2}\frac{ds}{s}.
\end{aligned}
\end{equation}
Now, it is easy to see that there exists a constant $C_1$ which depends only on $d$ and $\beta$ such that 
\begin{equation}\label{3.19}
\frac{s^{-\frac{d}{2}+\beta}}{C_1}\leq\zeta_{\beta}(s)\leq C_1s^{-\frac{d}{2}+\beta},\quad for \quad 0<s<1/2.
\end{equation}
Therefore 
\begin{equation*}
K^{\alpha}_1(\beta;x,y)\leq \frac{C_{\alpha}C_1}{\Gamma(\beta)2^{\frac{d}{2}+\beta-1}}\int_{0}^{1/2}s^{-\frac{d}{2}+\beta}e^{-\frac{1}{4s}\arrowvert x-y\arrowvert^2}\frac{ds}{s}
\end{equation*}
Since $\frac{1}{8s}+\frac{1}{4}\leq \frac{1}{4s}$ for $0\leq s\leq\frac{1}{2}$, then
if $\arrowvert x-y\arrowvert\geq1$, we have $\frac{1}{8s}+\frac{\arrowvert x-y\arrowvert^2}{4}\leq \frac{\arrowvert x-y\arrowvert^2}{4s}$ for $0\leq s\leq\frac{1}{2}$.
Thus, the last expression is bounded up to a constant by
\begin{equation*}
e^{-\frac{1}{4}\arrowvert x-y\arrowvert^2}\int_{0}^{1/2}s^{-\frac{d}{2}+\beta}e^{-\frac{1}{8s}}\frac{ds}{s},
\end{equation*}
hence we obtain
\begin{equation}\label{P3}
K^{\alpha}_1(\beta;x,y)\leq Ce^{-\frac{1}{4}\arrowvert x-y\arrowvert^2}. 
\end{equation}
 Finally, we study the region $\arrowvert x-y\arrowvert\leq1$. By a change of variables, we obtain
 \begin{equation*}
 \int_{0}^{1/2}s^{-\frac{d}{2}+\beta}e^{-\frac{1}{4s}(x-y)^2}\frac{ds}{s}=\frac{1}{\arrowvert x-y\arrowvert^{d-2\beta}}\int_{2\arrowvert x-y\arrowvert^2}^{\infty}s^{d/2-\beta}e^{-\frac{s}{4}}\frac{ds}{s}.
 \end{equation*}
In the case $\beta<\frac{d}{2}$, we get
\begin{equation}\label{P4}
K^{\alpha}_1(\beta;x,y)\leq\frac{C}{\arrowvert x-y\arrowvert^{d-2\beta}}.
\end{equation}
For the case $\beta=\frac{d}{2}$.
\begin{equation*}
\begin{aligned}
\int_{2\arrowvert x-y\arrowvert^2}^{\infty}e^{-\frac{s}{4}}\frac{ds}{s}&=\int_{2\arrowvert x-y\arrowvert^2}^{2}\frac{ds}{s}+\int_{2}^{\infty}e^{-\frac{s}{4}}\frac{ds}{s}\\
&=\log\left(\frac{1}{\arrowvert x-y\arrowvert}\right)+\int_{2}^{\infty}e^{-\frac{s}{4}}\frac{ds}{s}\\
&\leq 2\log\left(\frac{e}{\arrowvert x-y\arrowvert}\right)\left(1+\int_{2}^{\infty}e^{-\frac{s}{4}}\frac{ds}{s}\right).
\end{aligned}
\end{equation*}
Then, 
\begin{equation}\label{P5}
K^{\alpha}_1(\beta;x,y)\leq C\log\left(\frac{e}{\arrowvert x-y\arrowvert}\right).
\end{equation}
Finally, when $\beta>\frac{d}{2}$.
\begin{equation*}
\begin{aligned}
\int_{2\arrowvert x-y\arrowvert^2}^{\infty}s^{d/2-\beta}e^{-\frac{s}{4}}\frac{ds}{s}&\leq \int_{2\arrowvert x-y\arrowvert^2}^{2}s^{d/2-\beta}\frac{ds}{s}+\int_{2}^{\infty}s^{\frac{d}{2}-\beta}e^{-\frac{s}{4}}\frac{ds}{s}\\
&\leq \frac{C}{\arrowvert x-y\arrowvert^{2\beta-d}}.
\end{aligned}
\end{equation*}
Thus in this case we obtain 
\begin{equation}\label{P6}
K^{\alpha}_1(\beta;x,y)\leq C.
\end{equation}
We deduce from \eqref{P1}, \eqref{P3}, \eqref{P4}, \eqref{P5}, and \eqref{P6}, if we define for $x\in\mathbb{R}_{+}^{d}$,
\begin{equation*}
\Phi^{\alpha}(\beta;x)=\left\{\begin{array}{lr}
\frac{\mathcal{X}_{\arrowvert x\arrowvert<1}(x)}{\arrowvert x\arrowvert^{d-2\beta}}+e^{-\frac{1}{4}\arrowvert x\arrowvert^2}\mathcal{X}_{\arrowvert x\arrowvert\geq 1}(x),\quad\quad\quad\quad\quad\qquad if\,\,\beta<\frac{d}{2},\\
\log\left(\frac{e}{\arrowvert x\arrowvert}\right)\mathcal{X}_{\arrowvert x\arrowvert<1}(x)+e^{-\frac{1}{4}\arrowvert x\arrowvert^2}\mathcal{X}_{\arrowvert x\arrowvert\geq 1}(x),\quad\quad\quad if\,\,\beta=\frac{d}{2},\\
\mathcal{X}_{\arrowvert x\arrowvert<1}(x)+e^{-\frac{1}{4}\arrowvert x\arrowvert^2}\mathcal{X}_{\arrowvert x\arrowvert\geq 1}(x),\quad\quad\quad\quad\quad\quad\,\,\,if\,\,\beta>\frac{d}{2}.
\end{array}
\right.
\end{equation*}
we have proved \eqref{C1.1}.
\end{proof}
\begin{remark}\label{R2.2}
For every $\beta>0$, since the kernel $K^{\alpha}(\beta;x,y)$ is bounded by the integrable function $\Phi^{\alpha}(\beta;x-y)$, we obtain that the operator $T_{\alpha}^{\beta}$ is bounded on $L^p({\mathbb{R}_{+}^d})$ for all $p\in[1,\infty]$. Now, by Proposition \ref{FK}, we can proceed to prove that $T_{\alpha}^{\beta}f$ equals to \eqref{R3.1}. It suffices to show
\begin{equation*}
\langle T_{\alpha}^{\beta}f,\varphi_{\mu}^{\alpha} \rangle=(4\arrowvert \mu\arrowvert_1+2\arrowvert\alpha\arrowvert_1+2)^{-\beta}\langle f,\varphi_{\mu}^{\alpha}\rangle.
\end{equation*}
By Proposition \ref{FK} and H\"{o}lder's inequality, 
\begin{equation}
\begin{aligned}
\int_{\mathbb{R}_{+}^d}&\int_{\mathbb{R}_{+}^d}\arrowvert K^{\alpha}(\beta;x,y)\arrowvert \arrowvert f(y)\arrowvert\arrowvert\varphi_{\mu}^{\alpha}(x)\arrowvert dydx\\
&\leq C\int_{\mathbb{R}_{+}^d}\int_{\mathbb{R}_{+}^d}\arrowvert \Phi^{\alpha}(\beta;x-y)\arrowvert \arrowvert f(y)\arrowvert\arrowvert\varphi_{\mu}^{\alpha}(x)\arrowvert dydx\\
&\leq C\lVert f\lVert_{L^2(\mathbb{R}_{+}^d)}\lVert \varphi_{\mu}^{\alpha}\lVert_{L^2(\mathbb{R}_{+}^d)}.
\end{aligned}
\end{equation}
Therefore, by Fubini's theorem,
\begin{equation}
\begin{aligned}
\langle T_{\alpha}^{\beta}f,\varphi_{k}^{\alpha} \rangle&=\int_{\mathbb{R}_{+}^d}\int_{\mathbb{R}_{+}^d} K^{\alpha}(\beta;x,y) f(y)dy\overline{\varphi_{\mu}^{\alpha}(x)}dx\\
&=\int_{\mathbb{R}_{+}^d} f(y)\int_{\mathbb{R}_{+}^d} K^{\alpha}(\beta;x,y)\overline{\varphi_{\mu}^{\alpha}(x)}dxdy\\
&=(4\arrowvert \mu\arrowvert_1+2\arrowvert\alpha\arrowvert_1+2)^{-\beta}\langle f,\varphi_{\mu}^{\alpha}\rangle.
\end{aligned}
\end{equation}
\end{remark}
From Remark \ref{R2.2}, we will no longer distinguish between $T_{\alpha}^{\beta}$ and $(1+\mathcal{L}^{\alpha})^{-\beta}$ hereafter. Now we are ready to prove Lemma \ref{weighted}.\\
\par \textbf{Proof of Lemma \ref{weighted}.}  We will see that the kernel of $(1+\arrowvert x\arrowvert)^{2\beta}(1+\mathcal{L}^{\alpha})^{-\beta}$ satisfies
\begin{equation}\label{P1.1}
(1+ \arrowvert x\arrowvert)^{2\beta}\int_{\mathbb{R}_+^d}K^{\alpha}(\beta;x,y)dy\leq C
\end{equation}
and
\begin{equation}\label{P1.2}
\int_{\mathbb{R}_+^d}(1+ \arrowvert x\arrowvert)^{2\beta}K^{\alpha}(\beta;x,y)dx\leq C,
\end{equation}
where the constant $C$ depends only on $\alpha$ and $\beta$. Thus , $(1+ x)^{2\beta}(1+\mathcal{L}^{\alpha})^{-\beta}$ is bounded on $L^2(\mathbb{R}_{+}^d)$ by the Schur's test (see, for example, \cite[Theorem 6.18]{Folland} ).
\par We deal first with \eqref{P1.1}. From \eqref{C1.1} and \eqref{C1.2} we see that there exists a constant $C$ such that \eqref{P1.1} is valid for all $\arrowvert x\arrowvert<2$.
\par Assume $\arrowvert x\arrowvert<2\arrowvert x-y\arrowvert$. By using \eqref{KP} and estimate \eqref{I3} we obtain 
\begin{equation*}
\begin{aligned}
K^{\alpha}(\beta;x,y)&=\frac{1}{\Gamma(\beta)2^{\frac{d}{2}+\beta-1}}\int_{0}^{1}\zeta_{\beta}(s)e^{-\frac{1}{4}(s+\frac{1}{s})(\arrowvert x\arrowvert^2+\arrowvert y\arrowvert^2)}\prod_{i=1}^{d}\sqrt{\frac{1-s^2}{2s}x_iy_i}I_{\alpha_i}\left(\frac{1-s^2}{2s}x_iy_i\right)\frac{ds}{s}\\
&\leq\frac{1}{\Gamma(\beta)2^{\frac{d}{2}+\beta-1}}\int_{0}^{1}\zeta_{\beta}(s)e^{-\frac{1}{4}s\arrowvert x+y\arrowvert^2-\frac{1}{4s}\arrowvert x-y\arrowvert^2}\frac{ds}{s}\\
&\leq\frac{e^{-\frac{\arrowvert x-y\arrowvert^2}{8}}}{\Gamma(\beta)2^{\frac{d}{2}+\beta-1}}\int_{0}^{1}\zeta_{\beta}(s)e^{-\frac{1}{8}s\arrowvert x+y\arrowvert)^2-\frac{1}{8s}\arrowvert x-y\arrowvert^2}\frac{ds}{s}.
\end{aligned}
\end{equation*}
Therefore, for some constant $C$,
\begin{equation*}
\begin{aligned}
\int_{\{0\leq \arrowvert x\arrowvert<2\arrowvert x-y\arrowvert\}}(1+\arrowvert x\arrowvert)^{2\beta}K^{\alpha}(\beta;x,y)dy\leq\frac{1}{\Gamma(\beta)2^{\frac{d}{2}+\beta-1}}\int_{\mathbb{R}^d}(1+\arrowvert x-y\arrowvert)^{2\beta}e^{-\frac{\arrowvert x-y\arrowvert^2}{8}}\int_{0}^{1}\zeta_{\beta}(s)e^{-\frac{1}{8s}(x-y)^2}\frac{ds}{s}dy,
\end{aligned}
\end{equation*}
which is bounded by a constant independent of $x$.
\par It remains to consider the integral \eqref{P1.1} restricted to the set $E_x=\{y:\arrowvert x\arrowvert>2\arrowvert x-y\arrowvert\}$ when $\arrowvert x\arrowvert>2$. Observe that in this part, due to the identity $\arrowvert x+y\arrowvert^2=2\arrowvert x\arrowvert^2-\arrowvert x-y\arrowvert^2+2\arrowvert y\arrowvert^2$, we have 
\begin{equation}\label{P3.1}
\arrowvert x\arrowvert< \arrowvert x+y\arrowvert.
\end{equation}
As in the proof of Proposition \ref{FK}, we consider $K^{\alpha}(\beta;x,y)=K^{\alpha}_1(\beta;x,y)+K^{\alpha}_2(\beta;x,y)$. Then, by using \eqref{P3.1} and \eqref{P1} we obtain
\begin{equation*}
(1+\arrowvert x\arrowvert)^{2\beta}\int_{E_x}K^{\alpha}_2(\beta;x,y)dy\leq(1+\arrowvert x\arrowvert)^{2\beta}e^{-\frac{1}{8}\arrowvert x\arrowvert^2}\int_{\mathbb{R}^{d}}e^{-\frac{1}{4}\arrowvert x-y\arrowvert^2}dy\leq C,
\end{equation*}
where in the last inequality we have used that for each positive $b$, there exists a constant $C_b$ such that $(1+\arrowvert x\arrowvert)^{b}e^{-\frac{1}{8}\arrowvert x\arrowvert^2}\leq C_b$. In order to handle $K^{\alpha}_1(\beta;x,y)$, from \eqref{3.18}, \eqref{3.19} and \eqref{P3.1}, after some changes of variables we obtain
\begin{equation*}
\begin{aligned}
\int_{E_x}K^{\alpha}_1(\beta;x,y)dy&\leq C\int_{E_x}\int_{0}^{1/2}s^{-d/2+\beta}e^{-\frac{1}{4}\left(s\arrowvert x\arrowvert^2+\frac{1}{s}\arrowvert x-y\arrowvert^2\right)}\frac{ds}{s}dy\\
&=C\arrowvert x\arrowvert^{d-2\beta}\int_{E_x}\int_{0}^{\arrowvert x\arrowvert^2/2}u^{-d/2+\beta}e^{-\frac{1}{4}\left(u+\frac{1}{u}(\arrowvert x\arrowvert\arrowvert x-y\arrowvert)^2\right)}\frac{du}{u}dy\quad u=s\arrowvert x\arrowvert^2\\
&=C\arrowvert x\arrowvert^{d-2\beta}\int_{0}^{\arrowvert x\arrowvert/2}\int_{0}^{\arrowvert x\arrowvert^2/2}u^{-d/2+\beta}e^{-\frac{1}{4}\left(u+\frac{1}{u}(\arrowvert x\arrowvert t)^2\right)}\frac{du}{u}t^{d-1}dt\quad t=\arrowvert x-y\arrowvert\\
&=C\arrowvert x\arrowvert^{-2\beta}\int_{0}^{\arrowvert x\arrowvert^2/2}\int_{0}^{\arrowvert x\arrowvert^2/2}u^{-d/2+\beta}e^{-\frac{1}{4}\left(u+\frac{r^2}{u}\right)}\frac{du}{u}r^{d-1}dr.\quad r=\arrowvert x\arrowvert t
\end{aligned}
\end{equation*}
The last double integral is bounded by
\begin{equation*}
\begin{aligned}
\int_{0}^{\infty}\int_{0}^{\infty}u^{-d/2+\beta}e^{-\frac{1}{4}\left(u+\frac{r^2}{u}\right)}\frac{du}{u}r^{d-1}dr&=\int_{0}^{\infty}u^{-d/2+\beta-1}e^{-\frac{u}{4}}\int_{0}^{\infty}e^{-\frac{r^2}{4u}}r^{d-1}drdu\\
&=\int_{0}^{\infty}u^{\beta-1}e^{-\frac{u}{4}}du\int_{0}^{\infty}e^{-\frac{r^2}{4}}r^{d-1}dr,
\end{aligned}
\end{equation*}
and both integrals clearly converge. Since $\arrowvert x\arrowvert>2$, we have $(1+\arrowvert x\arrowvert)^{2\beta}\leq4^{\beta}\arrowvert x\arrowvert^{2\beta}$, thus we obtain
\begin{equation*}
(1+\arrowvert x\arrowvert)^{2\beta}\int_{E_x}K^{\alpha}_1(\beta;x,y)dy\leq C.
\end{equation*}
Finally we shall prove \eqref{P1.2}. We split
\begin{equation}\label{P1.4}
\int_{\mathbb{R}_+^d}(1+ \arrowvert x\arrowvert)^{2\beta}K^{\alpha}(\beta;x,y)dx=\int_{\{\arrowvert y\arrowvert>2\arrowvert x-y\arrowvert\}}(1+ \arrowvert x\arrowvert)^{2\beta}K^{\alpha}(\beta;x,y)dx+\int_{\{\arrowvert y\arrowvert<2\arrowvert x-y\arrowvert\}}(1+ \arrowvert x\arrowvert)^{2\beta}K^{\alpha}(\beta;x,y)dx.
\end{equation}
Since $\arrowvert x\arrowvert\leq\frac{3}{2}\arrowvert y\arrowvert$ when $\arrowvert y\arrowvert>2\arrowvert x-y\arrowvert$ and the kernel $K^{\alpha}(\beta;x,y)$ is symmetric, the first integral of \eqref{P1.4} is less than or equal to 
\begin{equation*}
\left(\frac{3}{2}\right)^{\beta}(1+\arrowvert y\arrowvert)^{2\beta}\int_{\{\arrowvert y\arrowvert>2\arrowvert x-y\arrowvert\}}K^{\alpha}(\beta;y,x)dx,
\end{equation*}
which can be bounded, in the same way as \eqref{P1.1}, by a constant depending only on $d$, $\alpha$ and $\beta$. For the second term of \eqref{P1.4}, by estimate \eqref{C1.1} and the expression \eqref{C1.2}, we have 
\begin{equation*}
\int_{\{\arrowvert y\arrowvert<2\arrowvert x-y\arrowvert<4\}}(1+ \arrowvert x\arrowvert)^{2\beta}K^{\alpha}(\beta;x,y)dx\leq C\int_{\mathbb{R}^d}\Phi^{\alpha}(\beta;x-y)dx\leq C.
\end{equation*}
On the other hand, since $\arrowvert x\arrowvert\leq 3\arrowvert x-y\arrowvert$ when $\arrowvert y\arrowvert<2\arrowvert x-y\arrowvert$, again by estimate \eqref{C1.1} and the expression \eqref{C1.2}, we obtain
\begin{equation*} 
\begin{aligned}
\int_{\{\arrowvert y\arrowvert<2\arrowvert x-y\arrowvert,2< \arrowvert x-y\arrowvert \}}(1+ \arrowvert x\arrowvert)^{2\beta}K^{\alpha}(\beta;x,y)dx&\leq C\int_{\{2< \arrowvert x-y\arrowvert\}}(1+\arrowvert x-y\arrowvert)^{2\beta}e^{-\frac{\arrowvert x-y\arrowvert^2}{4}}dx\\
&\leq C\int_{\{2< \arrowvert x\arrowvert\}}(1+\arrowvert x\arrowvert)^{2\beta}e^{-\frac{\arrowvert x\arrowvert^2}{4}}dx\leq C.
\end{aligned}
\end{equation*}
Then, we get the desired estimate \eqref{P1.2}. Thus, we have proved Lemma \ref{weighted}. $\hfill{\Box}$
\subsection{Trace lemma for the Laguerre operators}
In this subsection, we establish the following trace lemma, which will be used to prove the square function estimate (Proposition \ref{SP} below).
\begin{lemma}\label{Trace}
For $\beta>1$, there exists a constant $C>0$ such that the estimate
\begin{equation}\label{Trace1.1}
\lVert \mathcal{P}_n^{\alpha}f\rVert_{L^2(\mathbb{R}_{+}^{d})\rightarrow L^2(\mathbb{R}_{+}^{d}, (1+\arrowvert x\arrowvert)^{-\beta})}\leq Cn^{-\frac{1}{4}}
\end{equation}
holds for every $n\in \mathbb{N}$. 
\end{lemma}
Let $\mathscr{L}_{n}^{a}$ denote the normalized Laguerre function of type $a$, defined by
\begin{equation}\label{Define}
\mathscr{L}_{n}^{\alpha}(x)=\left(\frac{\Gamma(n+1)}{\Gamma(n+\alpha+1)}\right)^{\frac{1}{2}}e^{-\frac{x}{2}}x^{\frac{\alpha}{2}}L_{n}^{\alpha}(x).
\end{equation}
To prove Lemma \ref{Trace}, we use the following result.
\begin{lemma}\label{L}\cite[Theorem 1.5.3]{TS3}. Let $\nu=4n+2a+2$ and $a>-1$.
\begin{equation}\label{L1.1}
\arrowvert\mathscr{L}_{n}^{a}(x)\arrowvert\leq C\left\{\begin{array}{lr}
(x\nu)^{a/2},\qquad\qquad\qquad\quad\quad\qquad\,\,\,\, 0\leq x\leq 1/\nu,\\
(x\nu)^{-1/4}, \qquad\qquad\quad\qquad\quad\qquad\,\, 1/\nu\leq x\leq \nu/2,\\
\nu^{-1/4}(\nu^{1/3}+\arrowvert\nu-x\arrowvert)^{-1/4},\qquad\quad\quad\nu/2\leq x\leq 3\nu/2,\\
e^{-\gamma x},\qquad\qquad\qquad\qquad\qquad\qquad\, x\geq 3\nu/2,
\end{array}
\right.
\end{equation}
where $\gamma>0$ is a constant. Moreover, if $1\leq x\leq\nu-\nu^{1/3}$, we have 
\begin{equation}\label{L1.2}
\mathscr{L}_{n}^{a}(x)=(2/\pi)^{\frac{1}{2}}(-1)^{n}x^{-\frac{1}{4}}(\nu-x)^{-\frac{1}{4}}\cos\left(\frac{\nu(2\theta-\sin 2\theta)-\pi}{4}\right)+O\left(\frac{\nu^{\frac{1}{4}}}{(\nu-x)^{\frac{7}{4}}}+(x\nu)^{-\frac{3}{4}}\right),
\end{equation}
where $\theta=\arccos(x^{1/2}\nu^{-1/2})$.
\end{lemma}
We now prove Lemma~\ref{Trace}. The approach follows a similar argument to \cite[Proof of Lemma 1.5]{CD}.
\begin{proof}
We begin with a local $L^2$ estimate for $\mathcal{P}_n^{\alpha}$. It suffices to prove \eqref{Trace1.1} by establishing
\begin{equation}\label{Proof1.1}
\int_{[0,M]^d} \left\lvert \mathcal{P}_n^{\alpha} f(x) \right\rvert^2  dx \leq C M n^{-\frac{1}{2}} \lVert f \rVert_2^2
\end{equation}
for every $M \geq 1$. Then \eqref{Trace1.1} follows by decomposing $\mathbb{R}_+^d$ into dyadic shells and applying \eqref{Proof1.1} to each, using $\beta > 1$.

To prove \eqref{Proof1.1}, we decompose the Laguerre expansion as follows:
\begin{equation}\label{Proof1.2}
f = \sum_{i=1}^{d} f_i(x)
\end{equation}
where $f_1, \dots, f_d$ are pairwise orthogonal and, for each $1 \leq i \leq d$, $\mu_i \geq |\mu|/d$ whenever $\langle f_i, \varphi_\mu^{\alpha} \rangle \neq 0$ (see \cite{CD}). Note that $\mu_i \sim |\mu|$ when $\langle f_i, \varphi_\mu^{\alpha} \rangle \neq 0$. Thus, to establish \eqref{Proof1.1}, it suffices to show for each $1 \leq i \leq d$ and $M > 0$ that
\begin{equation}\label{Proof1.3}
\int_{[0,M]^d} \left\lvert \mathcal{P}_n^{\alpha} f_i(x) \right\rvert^2  dx \leq C M \sum_{|\mu|_1 = n} \mu_i^{-\frac{1}{2}} \left\lvert \langle f_i, \varphi_\mu^{\alpha} \rangle \right\rvert^2.
\end{equation}
By symmetry, we need only prove \eqref{Proof1.3} for $i=1$. Set $c(\mu) = \langle f_1, \varphi_\mu^{\alpha} \rangle$. Then
\begin{equation}
\left\lvert \mathcal{P}_n^{\alpha} f_1(x) \right\rvert^2 = \sum_{|\mu|_1 = n} \sum_{|\nu|_1 = n} c(\mu) \overline{c(\nu)} \prod_{i=1}^{d} \varphi_{\mu_i}^{\alpha_i}(x_i) \varphi_{\nu_i}^{\alpha_i}(x_i).
\end{equation}
By Fubini's theorem,
\begin{equation}
\int_{[0,M]^d} \left\lvert \mathcal{P}_n^{\alpha} f_1(x) \right\rvert^2  dx \leq \sum_{|\mu|_1 = n} \sum_{|\nu|_1 = n} c(\mu) \overline{c(\nu)}  \int_0^{M} \varphi_{\mu_1}^{\alpha_1}(x_1) \varphi_{\nu_1}^{\alpha_1}(x_1)  dx_1 \prod_{i=2}^{d}  \langle \varphi_{\mu_i}^{\alpha_i}, \varphi_{\nu_i}^{\alpha_i} \rangle .
\end{equation}
Since the $\varphi_{\mu_i}^{\alpha_i}$ are orthogonal, $\mu_i = \nu_i$ for $i=2,\dots,d$ whenever $\langle \varphi_{\mu_i}^{\alpha_i}, \varphi_{\nu_i}^{\alpha_i} \rangle \neq 0$. This implies $\mu_1 = \nu_1$ as $|\mu|_1 = |\nu|_1 = n$. Thus,
\begin{equation}
\int_{[0,M]^d} \left\lvert \mathcal{P}_n^{\alpha} f_1(x) \right\rvert^2  dx \leq \sum_{|\mu|_1 = n} |c(\mu)|^2 \int_0^{M} \left\lvert \varphi_{\mu_1}^{\alpha_1}(x_1) \right\rvert^2  dx_1.
\end{equation}
To obtain the desired estimate, it suffices to show
\begin{equation}\label{int_est}
\int_0^{M} \left\lvert \varphi_{\mu_1}^{\alpha_1}(x) \right\rvert^2  dx \leq C M \mu_1^{-1/2}.
\end{equation}
When $\mu_1 < M^2$, the estimate holds trivially since $\lVert \varphi_{\mu_1}^{\alpha_1} \rVert_2 = 1$. For $\mu_1 > M^2$, since $\varphi_{\mu_1}^{\alpha_1}(x) = \sqrt{2x} \mathscr{L}_{\mu_1}^{\alpha_1}(x^2)$, Lemma~\ref{L} yields a constant $C > 0$ such that $|\varphi_{\mu_1}^{\alpha_1}(x)| \leq C \mu_1^{-1/4}$ for all $x \in [0, M]$. Thus
\[
\int_0^{M} \left\lvert \varphi_{\mu_1}^{\alpha_1}(x) \right\rvert^2  dx \leq C \mu_1^{-1/2} \int_0^M  dx = C M \mu_1^{-1/2},
\]
completing the proof of Lemma~\ref{Trace}.
\end{proof}
To make the estimate \eqref{Trace1.1} applicable to the proof of the square function estimate, we require an extended version of \eqref{Trace1.1}. For any function $F$ supported on $[0,1]$ and for $2 \leq q < \infty$, we define (see \cite[(2.6)]{CD})
\begin{equation*}
\lVert F\rVert_{N^2,q}:=\left(\frac{1}{N^2}\sum_{i=1}^{N^2}\sup_{\xi\in[\frac{i-1}{N^2},\frac{i}{N^2})}\arrowvert F(x)\arrowvert ^{q}\right)^{1/q},\quad\quad N\in \mathbb{N}.
\end{equation*}
Originally appearing in \cite{CSA, DOS} in the context of spectral multipliers, the norm $\lVert F\rVert_{N^{2},q}$ was developed to handle compact self-adjoint operators cases. By using the same argument as in \cite{CD}, we obtain the following extended result.
\begin{lemma}\label{ETrace}
For $N\in \mathbb{N}$, let $F$ be a function supported in $[N/4,N]$ and for all $f\in L^2(\mathbb{R}_{+}^d)$. 
If $\beta>1$, then there exists a constant $C_\beta$ independent of $F$ and $f$ such that
\begin{equation}\label{ETrace1}
\int_{\mathbb{R}_{+}^{d}}\arrowvert F(\sqrt{\mathcal{L}^{\alpha}})f(x)\arrowvert^2(1+\arrowvert x\arrowvert)^{-\beta}dx\leq C_\beta N\lVert \delta_NF\rVert_{N^2,2}^2\int_{\mathbb{R}_{+}^{d}}\arrowvert f(x)\arrowvert ^2dx.
\end{equation}
If $0<\beta\leq1$, then for every $\varepsilon>0$, there exists a constant $C_{\beta,\varepsilon}$ independent of $F$ and $f$ such that
\begin{equation}\label{ETrace2}
\int_{\mathbb{R}_{+}^{d}}\arrowvert F(\sqrt{\mathcal{L}^{\alpha}})f(x)\arrowvert^2(1+\arrowvert x\arrowvert)^{-\beta}dx\leq C_{\beta,\varepsilon} N^{\frac{\beta}{1+\varepsilon}}\lVert \delta_NF\rVert_{N^2,2\beta^{-1}(1+\varepsilon)}^2\int_{\mathbb{R}_{+}^{d}}\arrowvert f(x)\arrowvert ^2dx.
\end{equation}
Where $\delta_{N}F(x)$ is defined by $F(Nx)$.
\end{lemma}
\begin{proof} As the proof in \cite{CD} extends to more general cases involving compact operators with spectral projection estimation \eqref{Trace1.1}, we omit the details. Crucially, the validity of \eqref{ETrace2} follows by combining \eqref{ETrace1} with a bilinear interpolation theorem. For full technical discussion, see \cite[pp. 12–15]{CD}; we omit the derivation here.
\end{proof}
\section{Proof of Theorem \ref{T1.2} }\label{Section3}
In this section, we reduce Theorem \ref{T1.2} to showing a square function estimate.
\subsection{Square function estimate}
We begin by recalling that
\begin{equation}\label{SI}
S_{*}^{\lambda}(\mathcal{L}^{\alpha})f(x)\leq C_{\lambda,\rho}\sup_{0<R<\infty}\left(\frac{1}{R}\int_0^{R}\arrowvert S_t^{\rho}(\mathcal{L}^{\alpha})f(x)\arrowvert^2dt\right)^{1/2}
\end{equation}
provided that $\rho>-1/2$ and $\lambda>\rho +1/2$. This was shown in \cite[pp.278-279]{SW} (see also \cite[p.13]{CD}). Via a dyadic decomposition, we write $x_{+}^{\rho}=\sum_{k\in\mathbb{Z}}2^{-k\rho}\phi^{\rho}(2^kx)$ for some $\phi^{\rho}\in C_{c}^{\infty}(2^{-3},2^{-1})$. Thus
\begin{equation*}
(1-\arrowvert \xi\arrowvert^2)_{+}^{\rho}=:\phi_0^{\rho}(\xi)+\sum_{k=1}^{\infty}2^{-k\rho}\phi_{k}^{\rho}(\xi),
\end{equation*}
where $\phi_{k}^{\rho}=\phi^{\rho}(2^k(1-\arrowvert \xi\arrowvert^2)), k\geq1$. We also note that supp\,$\phi_{0}^{\rho}\subseteq\{\xi:\arrowvert\xi\arrowvert\leq 7\times 2^{-3}\}$ and supp\,$\phi_{k}^{\rho}\subseteq\{\xi:1-2^{-1-k}\leq\arrowvert\xi\arrowvert\leq 1-2^{-k-3}\}$. Using \eqref{SI}, for $\lambda>\rho+1/2$ we have
\begin{equation}\label{S1}
S_{*}^{\lambda}(\mathcal{L}^{\alpha})f(x)\leq C\left(\sup_{0<R<\infty}\frac{1}{R}\int_0^{R}\arrowvert \phi_{0}^{\rho}(t^{-1}\sqrt{\mathcal{L}^{\alpha}})f(x)\arrowvert^2dt\right)^{1/2}+C\sum_{k=1}^{\infty}2^{-k\rho}\left(\int_0^{\infty}\arrowvert \phi_{k}^{\rho}(t^{-1}\sqrt{\mathcal{L}^{\alpha}})f(x)\arrowvert^2\frac{dt}{t}\right)^{1/2}.
\end{equation}
Next, we define the square function $\mathfrak{S}_{\delta}$ by
\begin{equation*}\label{S2}
\mathfrak{S}_{\delta}f(x)=\left(\int_0^{\infty}\arrowvert \phi\left(\delta^{-1}\left(1-\frac{\mathcal{L}^{\alpha}}{t^2}\right)\right)f(x)\arrowvert^2\frac{dt}{t}\right)^{1/2}\quad,\quad 0<\delta<\frac{1}{2},
\end{equation*}
where $\phi$ is a fixed $C^{\infty}$ function supported in $[2^{-3},2^{-1}]$ with $\arrowvert \phi\arrowvert\leq 1$.
The proof of the sufficiency part of Theorem \ref{T1.2} now reduces to proving the following proposition, which is a weighted $L^2$-estimate for the square function $\mathfrak{S}_{\delta}$.
\begin{proposition}\label{SP}
Fix $\alpha\in[-1/2,\infty)^d$, let $0<\delta<1/2$, $0<\varepsilon\leq1/2$, and let $0\leq\beta<d$. Then, there exists a constant $C>0$, independent of $f$ and $\delta$, such that
\begin{equation}\label{S3}
\int_{\mathbb{R}_{+}^{d}}\arrowvert\mathfrak{S}_{\delta}f(x)\arrowvert^2(1+\arrowvert x\arrowvert)^{-\beta}dx\leq C\delta A_{\beta,d}^{\varepsilon}(\delta)\int_{\mathbb{R}_{+}^{d}}\arrowvert f(x)\arrowvert^2(1+\arrowvert x\arrowvert)^{-\beta}dx,
\end{equation}
where
\begin{equation}\label{S4}
A_{\beta,d}^{\varepsilon}(\delta):=\left\{\begin{array}{lr}
\delta^{-\varepsilon},\qquad\quad 0\leq \beta\leq 1,\quad if \quad d=1,\\
\delta^{\frac{1}{2}-\frac{\beta}{2}},\,\,\qquad 1< \beta< d,\quad if\quad d\geq 2.
\end{array}
\right.
\end{equation}
\end{proposition}
Once we have the estimate \eqref{S3}, the proof of Theorem \ref{T1.2} follows from the argument outlined below. \\
\par\textbf{Proof of Theorem \ref{T1.2}.}
 Since $\lambda>0$, we choose $\eta>0$ (to be taken arbitrarily small later) such that $\lambda-\eta>0$. Setting $\rho=\lambda-\eta-\frac{1}{2}$, we can now apply \eqref{S1}. First, we estimate the right-hand side's first term in \eqref{S1}. Since $(1+\arrowvert x\arrowvert)^{-\beta}$ is an $A_2$ weight and the heat kernel of $\mathcal{L}^{\alpha}$ satisfies the Gaussian bound, the argument in \cite[Lemma 3.1]{CLSL} yields
\begin{equation}
\left\lVert \sup_{0<t<\infty}\arrowvert \phi_{0}^{\rho}(t^{-1}\sqrt{\mathcal{L}^{\alpha}})f\arrowvert \right\rVert_{L^2(\mathbb{R}_{+}^d,(1+\arrowvert x\arrowvert)^{-\beta})}\leq \lVert \mathscr{M} f\rVert_{L^2(\mathbb{R}_{+}^d,(1+\arrowvert x\arrowvert)^{-\beta})}\leq C\lVert f\rVert_{L^2(\mathbb{R}_{+}^d,(1+\arrowvert x\arrowvert)^{-\beta})}.
\end{equation} 
where $\mathscr{M}$ is the Hardy-Littlewood maximal operator. Thus, it suffices to show that the operator
\begin{equation*}
\sum_{k=1}^{\infty}2^{-k\rho}\left(\int_0^{\infty}\arrowvert \phi_{k}^{\rho}(t^{-1}\sqrt{\mathcal{L}^{\alpha}})f(x)\arrowvert^2\frac{dt}{t}\right)^{1/2}
\end{equation*}
is bounded on $L^2(\mathbb{R}_{+}^d,(1+\arrowvert x\arrowvert)^{-\beta})$. If $d= 1$, then using Minkowski's inequlaity and the estimate \eqref{S3} with $\delta=2^{-k}$, we obtain, by our choice of $\rho$,
\begin{equation*}
\begin{aligned}
&\left\lVert\sum_{k=1}^{\infty}2^{-k\rho}\left(\int_0^{\infty}\arrowvert \phi_{k}^{\rho}(t^{-1}\sqrt{\mathcal{L}^{\alpha}})f(x)\arrowvert^2\frac{dt}{t}\right)^{1/2}\right\rVert_{L^2(\mathbb{R}_{+}^d,(1+\arrowvert x\arrowvert)^{-\beta})}\\
&\leq C\sum_{k=1}^{\infty}2^{-k(\lambda-\eta-\varepsilon/2)}\lVert f\rVert_{L^2(\mathbb{R}_{+}^d,(1+\arrowvert x\arrowvert)^{-\beta})}.
\end{aligned}
\end{equation*}
By taking $\eta$ and $\varepsilon$ sufficiently small, the right-hand side is bounded by $C\lVert f\rVert_{L^2(\mathbb{R}_{+}^d,(1+\arrowvert x\arrowvert)^{-\beta})}$. This establishes the desired boundedness of $S_{*}^{\lambda}(\mathcal{L}^{\alpha})$ on $ L^2(\mathbb{R}_{+}^d,(1+\arrowvert x\arrowvert)^{-\beta})$ for $d=1$. The same argument extends to the case $ d\geq2$ and $\beta>1$, and the range of $\beta$ can be extended to $\beta>0$ by interpolation. For the interpolation argument, one can refer to \cite[p.18]{CD}.$\hfill{\Box}$\\
\par Now, to complete the proof of Theorem \ref{T1.2}, it remains to prove Proposition \ref{SP}.
\subsection{Proof of the square function estimate}
In this subsection, we establish Proposition \ref{SP}. To this end, we decompose $\mathfrak{S}_{\delta}$ into high- and low-frequency parts. For every fixed $\alpha\in[-1/2,\infty)^{d}$, let us set 
\begin{equation*}
\mathfrak{S}_{\delta}^{l}f(x)=\left(\int_{1/2}^{\delta^{-1/2}}\arrowvert \phi\left(\delta^{-1}\left(1-\frac{\mathcal{L}^\alpha}{t^2}\right)\right)f(x)\arrowvert^2\frac{dt}{t}\right)^{1/2},
\end{equation*}
\begin{equation*}
\mathfrak{S}_{\delta}^{h}f(x)=\left(\int_{\delta^{-1/2}}^{\infty}\arrowvert \phi\left(\delta^{-1}\left(1-\frac{\mathcal{L}^\alpha}{t^2}\right)\right)f(x)\arrowvert^2\frac{dt}{t}\right)^{1/2}.
\end{equation*}
Since the first eigenvalue of the Laguerre operator $\mathcal{L}^\alpha$ is larger than or equal to 1 for $\alpha\in[-1/2,\infty)^{d}$, $\phi\left(\delta^{-1}\left(1-\frac{\mathcal{L}^\alpha}{t^2}\right)\right)$=0 if $t\leq1$ because supp\,$\phi\subset(2^{-3},2^{-1})$ and $\delta\leq1/2$. Thus, we have
$$\mathfrak{S}_{\delta}f(x)\leq\mathfrak{S}_{\delta}^{l}f(x)+\mathfrak{S}_{\delta}^{h}f(x).$$
\par Now, in order to prove Proposition \ref{SP}, it is sufficient to show the following lemma.
\begin{lemma}\label{SL}
Let $A_{\beta,d}^{\varepsilon}(\delta)$ be given by \eqref{S4}. Then, for all $0<\delta<1/2$ and $0<\varepsilon\leq1/2$, we have the following estimates:
\begin{equation}\label{low}
\int_{\mathbb{R}_{+}^{d}}\arrowvert\mathfrak{S}_{\delta}^{l}f(x)\arrowvert^2(1+\arrowvert x\arrowvert)^{-\beta}dx\leq C\delta A_{\beta,d}^{\varepsilon}(\delta)\int_{\mathbb{R}_{+}^{d}}\arrowvert f(x)\arrowvert^2(1+\arrowvert x\arrowvert)^{-\beta}dx,
\end{equation}
\begin{equation}\label{high}
\int_{\mathbb{R}_{+}^{d}}\arrowvert\mathfrak{S}_{\delta}^{h}f(x)\arrowvert^2(1+\arrowvert x\arrowvert)^{-\beta}dx\leq C\delta A_{\beta,d}^{\varepsilon}(\delta)\int_{\mathbb{R}_{+}^{d}}\arrowvert f(x)\arrowvert^2(1+\arrowvert x\arrowvert)^{-\beta}dx.
\end{equation}
\end{lemma}
Both of the proofs of the estimates \eqref{low} and \eqref{high} rely on the generalized trace Lemma \ref{ETrace}. The primary distinction lies in the fact that, for the estimate in \eqref{low}, we also employ the estimate in \eqref{weighted}, which proves particularly effective for the low-frequency part. The choice of $\delta^{-\frac{1}{2}}$ in the definition of $\mathfrak{S}_{\delta}^{l}$ is made by optimizing the estimates which result from two different approaches, see \cite[Remark 3.2]{CD}. The proof of this lemma follows directly from a minor adaptation of the arguments presented in \cite[Proof of Lemma 3.1]{CD}; we include the arguments here for completeness. The fundamental strategy involves an iterative process of appropriate decompositions. We remark that the proof technique for the high-frequency part \eqref{high} was originally introduced in \cite{Carbery1} and has subsequently found applications in various related problems, including those studied in \cite{CD, CLSL}.
\subsection{Proof of \eqref{low}: low frequency part}
\begin{proof}
We begin with the Littlewood-Paley decomposition associated with the operator $\mathcal{L}^\alpha$. Fix a smooth function $\psi \in C^{\infty}(\mathbb{R})$ supported in $(\frac{1}{2}, 2)$ such that $\sum_{k=-\infty}^{\infty} \psi(2^{-k}x) = 1$ for all $x \in \mathbb{R} \backslash {0}$. Spectral theory yields
\begin{equation}\label{LPC}
\sum_{k} \psi_k(\sqrt{\mathcal{L}^\alpha})f := \sum_{k} \psi(2^{-k}\sqrt{\mathcal{L}^\alpha})f = f
\end{equation}
for any $f \in L^2(\mathbb{R}_{+}^{d})$. Using \eqref{LPC}, we obtain
\begin{equation}\label{A2}
\left\lvert \mathfrak{S}{\delta}^{l}f(x) \right\rvert^2 \leq C \sum_{0\leq k\leq 1-\log_2\sqrt{\delta}} \int_{2^{k-1}}^{2^{k+2}} \left\lvert \phi\left( \delta^{-1} \left(1 - \frac{\mathcal{L}^\alpha}{t^2}\right) \right) \psi_k(\sqrt{\mathcal{L}^\alpha})f(x) \right\rvert^2 \frac{dt}{t}
\end{equation}
for $f \in L^2(\mathbb{R}_{+}^{d}) \cap L^2(\mathbb{R}_{+}^{d}, (1 + |x|)^{-\beta} dx)$. To exploit the disjointness of the spectral supports of $ \phi\left( \delta^{-1} \left(1 - \frac{\mathcal{L}^\alpha}{t^2}\right) \right)$, we introduce an additional decomposition in $t$. For $k \in \mathbb{Z}$ and $i = 0, 1, \dots, i_0 = \lfloor 8/\delta \rfloor + 1$, define intervals
\begin{equation}\label{A3}
I_i = \left[ 2^{k-1} + i \cdot 2^{k-1}\delta, 2^{k-1} + (i+1) \cdot 2^{k-1}\delta \right],
\end{equation}
so that $[2^{k-1}, 2^{k+2}] \subset \bigcup_{i=0}^{i_0} I_i$. Define functions $\eta_i$ adapted to $I_i$ by
\begin{equation}\label{A4}
\eta_i(s) = \eta \left( i + \frac{2^{k-1} - s}{2^{k-1}\delta} \right),
\end{equation}
where $\eta \in C_{c}^{\infty}(-1, 1)$ satisfies $\sum_{i \in \mathbb{Z}} \eta(\cdot - i) = 1$. For notational convenience, set
$$\phi_{\delta}(s):=\phi(\delta^{-1}(1-s^2)).$$
Then for $t \in I_i$, we have $\phi_{\delta}(s/t) \eta_{i'}(s) \neq 0$ only if $i - i\delta - 3 \leq i' \leq i + i\delta + 3$; see \cite[p.20]{CD}. \
Thus we can deduce that (see \cite[p.~20]{CD} for details)
\begin{equation}\label{Bclaim}
\left\lvert \mathfrak{S}_{\delta}^{\ell}f(x) \right\rvert^2 \leq C \sum_{0\leq k\leq 1-\log_2\sqrt{\delta}} \sum_{i} \sum_{i'=i-10}^{i+10} \int_{I_i} \left\lvert \phi_{\delta}(t^{-1}\sqrt{\mathcal{L}^\alpha}) \eta_{i'}(\sqrt{\mathcal{L}^\alpha}) \psi_k(\sqrt{\mathcal{L}^\alpha})f(x) \right\rvert^2 \frac{dt}{t}.
\end{equation}
Now, we claim that for $1/2 \leq t \leq \delta^{-1/2}$,
\begin{equation}\label{Claim}
\int_{\mathbb{R}_+^d} \left\lvert \phi_{\delta}(t^{-1}\sqrt{\mathcal{L}^\alpha}) g(x) \right\rvert^2 (1 + |x|)^{-\beta}  dx \leq C A_{\beta,d}^{\varepsilon}(\delta) \int_{\mathbb{R}_+^d} \left\lvert (1 + \mathcal{L}^\alpha)^{-\beta/4} g(x) \right\rvert^2  dx.
\end{equation}
To prove this claim, we consider the equivalent estimate
\begin{equation}\label{Claim1}
\int_{\mathbb{R}_+^d} \left\lvert \phi_{\delta}(t^{-1}\sqrt{\mathcal{L}^\alpha}) (1 + \mathcal{L}^\alpha)^{\beta/4} g(x) \right\rvert^2 (1 + |x|)^{-\beta}  dx \leq C A_{\beta,d}^{\varepsilon}(\delta) \int_{\mathbb{R}_+^d} |g(x)|^2  dx.
\end{equation}
We first establish the estimate for the case $d \geq 2$. Let $N = 8[t] + 1$. Note that $\operatorname{supp} \phi_{\delta}(\cdot/t) \subset [N/4, N]$. By estimate \eqref{ETrace1} in Lemma~\ref{ETrace}, we obtain
\begin{equation*}
\begin{aligned}
&\int_{\mathbb{R}_+^d} \left\lvert \phi_{\delta}(t^{-1}\sqrt{\mathcal{L}^\alpha}) (1 + \mathcal{L}^\alpha)^{\beta/4} g(x) \right\rvert^2 (1 + |x|)^{-\beta}  dx \\
&\quad \leq C N \left\lVert \phi_{\delta}(t^{-1} N x) (1 + N^2 x^2)^{\beta/4} \right\rVert_{N^2,2}^2 \int_{\mathbb{R}_+^d} |g(x)|^2  dx.
\end{aligned}
\end{equation*}
Estimate (3.12) in \cite{CD} yields
\[
\left\lVert \phi_{\delta}(t^{-1} N x) (1 + N^2 x^2)^{\beta/4} \right\rVert_{N^2,2}^2 \leq C N^{\beta-2}.
\]
Thus, noting that $1/2 \leq t \leq \delta^{-1/2}$ and $\beta > 1$, we obtain
\begin{equation*}
\int_{\mathbb{R}_+^d} \left\lvert \phi_{\delta}(t^{-1}\sqrt{\mathcal{L}^\alpha}) (1 + \mathcal{L}^\alpha)^{\beta/4} g(x) \right\rvert^2 (1 + |x|)^{-\beta}  dx \leq C \delta^{(1-\beta)/2} \int_{\mathbb{R}_+^d} |g(x)|^2  dx,
\end{equation*}
which establishes \eqref{Claim1} and consequently \eqref{Claim} for $d \geq 2$. If $d=1$, let $0 \leq \beta < 1$ and set $N = 8[t] + 1$. Then for any $\varepsilon > 0$, by \eqref{ETrace2} in Lemma~\ref{ETrace}, we have
\begin{equation}
\begin{aligned}
&\int_{\mathbb{R}_+^d} \left\lvert \phi_{\delta}(t^{-1}\sqrt{\mathcal{L}^\alpha}) (1 + \mathcal{L}^\alpha)^{\beta/4} g(x) \right\rvert^2 (1 + |x|)^{-\beta}  dx \\
&\quad \leq C_{\varepsilon} N^{\frac{\beta}{1+\varepsilon}} \left\lVert \phi_{\delta}(t^{-1} N x) (1 + N^2 x^2)^{\beta/4} \right\rVert_{N^2, \frac{2(1+\varepsilon)}{\beta}}^2 \int_{\mathbb{R}_+^d} |g(x)|^2  dx.
\end{aligned}
\end{equation}
As before (see \cite[p.~22]{CD}), we have
\[
\left\lVert \phi_{\delta}(t^{-1} N x) (1 + N^2 x^2)^{\beta/4} \right\rVert_{N^2, \frac{2(1+\varepsilon)}{\beta}}^2 \leq C N^{\frac{\beta(\varepsilon-1)}{1+\varepsilon}}.
\]
Thus it follows that
\begin{equation*}
\int_{\mathbb{R}_+^d} \left\lvert \phi_{\delta}(t^{-1}\sqrt{\mathcal{L}^\alpha}) (1 + \mathcal{L}^\alpha)^{\beta/4} g(x) \right\rvert^2 (1 + |x|)^{-\beta}  dx \leq C \delta^{-\frac{\beta}{2(1+\varepsilon)}} \int_{\mathbb{R}_+^d} |g(x)|^2  dx.
\end{equation*}
This establishes \eqref{Claim} for $d=1$.
Now, combining \eqref{Claim} and \eqref{Bclaim}, we see that $\int_{\mathbb{R}_+^d} \left\lvert \mathfrak{S}_{\delta}^{\ell} f(x) \right\rvert^2 (1 + |x|)^{-\beta}  dx$ is bounded by
\begin{equation*}
C A_{\beta,d}^{\varepsilon}(\delta) \sum_{0 \leq k \leq 1 - \log_2\sqrt{\delta}} \sum_{i} \sum_{i'=i-10}^{i+10} \int_{I_i} \left\lvert \eta_{i'}(\sqrt{\mathcal{L}^\alpha}) \psi_k(\sqrt{\mathcal{L}^\alpha}) (1 + \mathcal{L}^\alpha)^{-\beta/4} f(x) \right\rvert^2 \frac{dt}{t}.
\end{equation*}
Since the length of each interval $I_i$ is comparable to $2^{k-1}\delta$, integrating over $t$ and using the disjointness of the spectral supports yields
\begin{equation}\label{3.16}
\int_{\mathbb{R}_+^d} \left\lvert \mathfrak{S}_{\delta}^{\ell} f(x) \right\rvert^2 (1 + |x|)^{-\beta}  dx \leq C \delta A_{\beta,d}^{\varepsilon}(\delta) \int_{\mathbb{R}_+^d} \left\lvert (1 + \mathcal{L}^\alpha)^{-\beta/4} f(x) \right\rvert^2  dx.
\end{equation}
As the operator $(1 + \mathcal{L}^\alpha)^{-\beta}$ is self-adjoint, combining \eqref{3.16} with the dual estimate of \eqref{weighted1.1} gives the desired estimate \eqref{low}.
\end{proof}
\subsection{Proof of \eqref{high}: high frequency part}
For any even function $F$ with $\hat{F}\in L^1(\mathbb{R})$ and supp\,$\hat{F}\subseteq[-r,r]$ we can deduce from Lemma \ref{Finit} that (see \cite[3.14, p.23]{CD} for detail)
\begin{equation}\label{B2}
supp\,K_{F(t\sqrt{\mathcal{L}^{\alpha}})}(x,y)\subset D(t^{-1}r).
\end{equation} 
\par Next, we begin with several decompositions as in \cite{CD}.  Fixing an even function $\vartheta\in C_{c}^{\infty}$ which is identically one on $\{\arrowvert s\arrowvert\leq 1\}$ and supported on $\{\arrowvert s\arrowvert\leq 2\}$, let us set $j_0=[-\log_2\delta]-1$ and $\zeta_{j_o}(s):=\vartheta(2^{-j_0}s)$ and $\zeta_{j}(s):=\vartheta(2^{-j}s)-\vartheta(2^{-j+1}s)$ for $j>j_0$. Then, we have $\sum_{j\geq j_0}\zeta_{j}(s)\equiv1,\forall s>0$.
Recalling that $\phi_{\delta}(s)=\phi(\delta^{-1}(1-s^2))$, for $j\geq j_0$, we set
\begin{equation}\label{B3}
\phi_{\delta,j}(s)=\frac{1}{2\pi}\int_{0}^{\infty}\zeta_j(u)\hat{\phi_{\delta}}(u)\cos(su)du.
\end{equation}
It can be verified that
\begin{equation}\label{B4}
\arrowvert\phi_{\delta,j}(s)\arrowvert\leq\left\{\begin{array}{lr}C_N2^{(j_0-j)N},\qquad\qquad\qquad \arrowvert s\arrowvert \in [1-2\delta,1+2\delta],\\
C_N2^{j-j_0}(1+2^j\arrowvert s-1\arrowvert)^{-N},\quad\quad otherwise,
\end{array}
\right.
\end{equation}
for any $N$ and all $j\geq j_0$ (see \cite[3.16, p.23]{CD}). By the Fourier inversion formula, we obtain
\begin{equation}\label{B5}
\phi_{\delta}(s)=\phi(\delta^{-1}(1-s^2))=\sum_{j\geq j_0}\phi_{\delta,j}(s),\quad s>0.
\end{equation}
By \eqref{B2}, we have 
\begin{equation}\label{B6}
supp\,K_{\phi_{\delta,j}(\sqrt{\mathcal{L}^{\alpha}}/t)}(x,y)\subset D(t^{-1}2^{j+1}):=\{(x,y)\in \mathbb{R}_{+}^{d}\times\mathbb{R}_{+}^{d}:\arrowvert x-y\arrowvert\leq c_02^{j+1}/t\}.
\end{equation}
Now from \eqref{LPC}, it follows that
\begin{equation}\label{B7}
\arrowvert \mathfrak{S}_{\delta}^{h}f(x)\arrowvert ^2\leq C \sum_{k\geq 1-\log_2\sqrt{\delta}}\int_{2^{k-1}}^{2^{k+2}}\arrowvert \phi\left(\delta^{-1}\left(1-\frac{\mathcal{L}^{\alpha}}{t^2}\right)\right)\psi_{k}(\mathcal{L}^{\alpha})f(x)\arrowvert^2\frac{dt}{t}.
\end{equation}
For $k\geq 1-\log_2\sqrt{\delta}$ and $j\geq j_0$, let us set
\begin{equation*}\label{B8}
E^{k,j}(t):=\left\langle\left\lvert\phi_{\delta,j}(t^{-1}\sqrt{\mathcal{L}^{\alpha}})\psi_{k}(\sqrt{\mathcal{L}^{\alpha}})f(x)\right\rvert^2,(1+\arrowvert x\arrowvert)^{-\beta} \right\rangle. 
\end{equation*}
Using \eqref{B5}, \eqref{B7} and Minkowski's inequality, we have
\begin{equation}\label{B9}
\int_{\mathbb{R}_{+}^{d}}\arrowvert \mathfrak{S}_{\delta}^{h}f(x)\arrowvert ^2(1+\arrowvert x\arrowvert)^{-\beta}dx\leq C \sum_{k\geq 1-\log_2\sqrt{\delta}}\left(\sum_{j\geq j_0}\left(\int_{2^{k-1}}^{2^{k+2}}E^{k,j}(t)\frac{dt}{t}\right)^{1/2}\right)^2.
\end{equation}
In order to make use the localization property \eqref{B6}, we need to decompose $\mathbb{R}_{+}^{d}$ into disjoint cubes of side length $c_02^{j-k+2}$. For a given $k\in \mathbb{Z}$, $j\geq j_0$, and $m=(m_1,\cdot\cdot\cdot,m_d)\in \mathbb{N}^d$, let us set 
\begin{equation*}\label{B10}
Q_{m}=\left[c_02^{j-k+2}(m_1-\frac{1}{2}), c_02^{j-k+2}(m_1+\frac{1}{2})\right)\times\cdot\cdot\cdot\times\left[c_02^{j-k+2}(m_d-\frac{1}{2}), c_02^{j-k+2}(m_d+\frac{1}{2})\right) \cap \mathbb{R}_{+}^{d},
\end{equation*}
which are disjoint dyadic cuboid with side length $\frac{1}{2}c_02^{j-k+2}$ or $c_02^{j-k+2}$. If $m_i\neq 0, i=1,\cdot\cdot\cdot,d$, they are cubes centered at $x_m:=c_02^{j-k+2}m$ with side length $c_02^{j-k+2}$. Clearly, $\mathbb{R}_{+}^{d}=\cup_{m\in\mathbb{N}^{d}}Q_m$. For each $m$, we define $\widetilde{Q_m}$ by setting
\begin{equation*}\label{B11}
\widetilde{Q_m}:=\bigcup_{m'\in\mathbb{Z}_{\geq 0}^{d}:dist(Q_{m'},Q_{m})\leq\sqrt{d}c_02^{j-k+3}}Q_{m'},
\end{equation*}
and denote
\begin{equation*}\label{B12}
M_0:=\{m\in\mathbb{Z}_{\geq 0}^{d}:Q_0\cap\widetilde{Q_m}\neq \emptyset\}.
\end{equation*}
For $t\in[2^{k-1},2^{k+2}]$ it follows by \eqref{B6} that $\chi_{Q_m}\phi_{\delta,j}(t^{-1}\sqrt{\mathcal{L}^{\alpha}})\chi_{Q_{m'}}=0$ if $\widetilde{Q_m}\cap Q_{m'}=\emptyset$ for every $j,k$. Hence, it is clear that
\begin{equation*}\label{B13}
\phi_{\delta,j}(t^{-1}\sqrt{\mathcal{L}^{\alpha}})\psi_{k}(\sqrt{\mathcal{L}^{\alpha}})f=\sum_{m,m':dist(Q_m,Q_{m'})<t^{-1}c_02^{j+2}}\chi_{Q_m}\phi_{\delta,j}(t^{-1}\sqrt{\mathcal{L}^{\alpha}})\chi_{Q_{m'}}\psi_{k}(\sqrt{\mathcal{L}^{\alpha}}),
\end{equation*}
which gives 
\begin{equation}\label{B14}
E^{k,j}(t)\leq C\sum_{m}\left\langle\left\lvert\chi_{Q_m}\phi_{\delta,j}(t^{-1}\sqrt{\mathcal{L}^{\alpha}})\chi_{\widetilde{Q_m}}\psi_{k}(\sqrt{\mathcal{L}^{\alpha}})f(x)\right\rvert^2,(1+\arrowvert x\arrowvert)^{-\beta} \right\rangle.
\end{equation}
To exploit orthogonality generated by the disjointness of spectral support, we further decompose $\phi_{\delta,j}$ which is not compactly supported. We choose an even function $\theta\in C_{c}^{\infty}(-4,4)$ such that $\theta(s)=1$ for $s\in(-2,2)$. Set
\begin{equation}\label{B15}
\varphi_{0,\delta}(s):=\theta(\delta^{-1}(1-s)),\quad\varphi_{\mathscr{l},\delta}(s):=\theta(2^{-\mathscr{l}}\delta^{-1}(1-s))-\theta(2^{-\mathscr{l}+1}\delta^{-1}(1-s))
\end{equation}
for all $\mathscr{l}\geq1$ such that $1=\sum_{\mathscr{l}=0}^{\infty}\varphi_{\mathscr{l},\delta}(s)$ and $\phi_{\delta,j}=\sum_{\mathscr{l}=0}^{\infty}(\varphi_{\mathscr{l},\delta}\phi_{\delta,j})(s)$ for all $s>0$. We put it into \eqref{B14} to write
\begin{equation}\label{B16}
\left(\int_{2^{k-1}}^{2^{k+2}}E^{k,j}(t)\frac{dt}{t}\right)^{1/2}\leq\sum_{\mathscr{l}=0}^{\infty}\left(\sum_{m}\int_{2^{k-1}}^{2^{k+2}}E_{m}^{k,j,\mathscr{l}}(t)\frac{dt}{t}\right)^{1/2},
\end{equation}
where 
\begin{equation*}
E_{m}^{k,j,\mathscr{l}}(t)=\left\langle\left\lvert\chi_{Q_m}(\varphi_{\mathscr{l},\delta}\phi_{\delta,j})(t^{-1}\sqrt{\mathcal{L}^{\alpha}})\chi_{\widetilde{Q_m}}\psi_{k}(\sqrt{\mathcal{L}^{\alpha}})f(x)\right\rvert^2,(1+\arrowvert x\arrowvert)^{-\beta} \right\rangle.
\end{equation*}
Recalling that \eqref{A3} and \eqref{A4}, we observe that, for every $t\in I_i$, it is possible that $\varphi_{\mathscr{l},\delta}(s/t)\eta_{i'}(s) \neq 0 $ only when $i-2^{\mathscr{l}+6}\leq i'\leq i+2^{\mathscr{l}+6}$. Hence,
\begin{equation*}
(\varphi_{\mathscr{l},\delta}\phi_{\delta,j})(t^{-1}\sqrt{\mathcal{L}^{\alpha}})=\sum_{i'=i-2^{\mathscr{l}+6}}^{i+2^{\mathscr{l}+6}}(\varphi_{\mathscr{l},\delta}\phi_{\delta,j})(t^{-1}\sqrt{\mathcal{L}^{\alpha}})\eta_{i'}(\sqrt{\mathcal{L}^{\alpha}}),\quad t\in I_i.
\end{equation*}
From this and Cauchy-Schwarz's inequality, we have
\begin{equation*}
E_{m}^{k,j,\mathscr{l}}(t)\leq C2^{\mathscr{l}}\sum_{i'=i-2^{\mathscr{l}+6}}^{i+2^{\mathscr{l}+6}}E_{m,i'}^{k,j,\mathscr{l}}(t)
\end{equation*}
for $t\in I_i$ where
\begin{equation*}
E_{m,i'}^{k,j,\mathscr{l}}(t)=\left\langle\left\lvert\chi_{Q_m}(\varphi_{\mathscr{l},\delta}\phi_{\delta,j})(t^{-1}\sqrt{\mathcal{L}^{\alpha}})\eta_{i'}(\sqrt{\mathcal{L}^{\alpha}})\left[\chi_{\widetilde{Q_m}}\psi_{k}(\sqrt{\mathcal{L}^{\alpha}})f\right](x)\right\rvert^2,(1+\arrowvert x\arrowvert)^{-\beta} \right\rangle.
\end{equation*}
Combining this with \eqref{B16}, we get
\begin{equation}\label{B17}
\left(\int_{2^{k-1}}^{2^{k+2}}E^{k,j}(t)\frac{dt}{t}\right)^{1/2}\leq C\sum_{\mathscr{l}=0}^{\infty}2^{\mathscr{l}/2}\left(\sum_{m}\sum_{i}\int_{I_i}\sum_{i'=i-2^{\mathscr{l}+6}}^{i+2^{\mathscr{l}+6}}E_{m,i'}^{k,j,\mathscr{l}}(t)\frac{dt}{t}\right)^{1/2}.
\end{equation}
To continue, we distinguish two cases: $j>k$; and $j\leq k$. The first case is more involved and we need to distinguish several sub-cases which we separately handle.\\
\par\textbf{Case $j>k$} From the above inequality  \eqref{B17} we have
\begin{equation*}
\left(\int_{2^{k-1}}^{2^{k+2}}E^{k,j}(t)\frac{dt}{t}\right)^{1/2}\leq I_1(j,k)+I_2(j,k)+I_3(j,k),
\end{equation*}
where 
\begin{equation}\label{B18}
I_1(j,k):=\sum_{\mathscr{l}=0}^{[-\log_2\delta]-3}2^{\mathscr{l}/2}\left(\sum_{m\in M_{0}}\sum_{i}\int_{I_i}\sum_{i'=i-2^{\mathscr{l}+6}}^{i+2^{\mathscr{l}+6}}E_{m,i'}^{k,j,\mathscr{l}}(t)\frac{dt}{t}\right)^{1/2},
\end{equation}
\begin{equation}\label{B19}
I_2(j,k):=\sum_{\mathscr{l}=[-\log_2\delta]-2}^{\infty}2^{\mathscr{l}/2}\left(\sum_{m\in M_{0}}\sum_{i}\int_{I_i}\sum_{i'=i-2^{\mathscr{l}+6}}^{i+2^{\mathscr{l}+6}}E_{m,i'}^{k,j,\mathscr{l}}(t)\frac{dt}{t}\right)^{1/2},
\end{equation}
\begin{equation}\label{B20}
I_3(j,k):=\sum_{\mathscr{l}=0}^{\infty}2^{\mathscr{l}/2}\left(\sum_{m\notin M_{0}}\sum_{i}\int_{I_i}\sum_{i'=i-2^{\mathscr{l}+6}}^{i+2^{\mathscr{l}+6}}E_{m,i'}^{k,j,\mathscr{l}}(t)\frac{dt}{t}\right)^{1/2}.
\end{equation}
We will now estimate $I_1(j,k)$, $I_2(j,k)$, and $I_3(j,k)$ one by one.\\
\\
\textit{Estimate of the term  $I_1(j,k)$.} We claim that, for any $N>0$,
\begin{equation}\label{B21}
I_1(j,k)\leq C_N 2^{(j_0-j)N}\left(\delta A_{\beta,d}^{\varepsilon}(\delta)\right)^{1/2}\left(\int_{\mathbb{R}_{+}^{d}}\arrowvert \psi_k(\sqrt{\mathcal{L}^{\alpha}})f(x)\arrowvert ^2(1+\arrowvert x\arrowvert)^{-\beta}dx\right)^{1/2}
\end{equation}
where $A_{\beta,d}^{\varepsilon}(\delta)$ is defined in \eqref{S4}. Let us first consider the case $d\geq2$. For \eqref{B21}, it suffices to show
\begin{equation}\label{B22}
E_{m,i'}^{k,j,\mathscr{l}}(t)\leq C_N2^{-\mathscr{l}N}2^{(j_0-j)N}2^{k}\delta \int_{\mathbb{R}_{+}^{d}}\arrowvert \eta_{i'}(\sqrt{\mathcal{L}^{\alpha}})[ \chi_{\widetilde{Q_m}}\psi_k(\sqrt{\mathcal{L}^{\alpha}})f](x)\arrowvert ^2d\mu_{\alpha}(x)
\end{equation}
for any $N>0$ while $t\in I_i$ being fixed and $i-2^{\mathscr{l}+6}\leq i'\leq i+2^{\mathscr{l}+6}$. Indeed, since the support of $\eta_{i}$ are boundedly overlapping, \eqref{B22} gives that
\begin{equation*}
\sum_{i}\int_{I_i}\sum_{i'=i-2^{\mathscr{l}+6}}^{i+2^{\mathscr{l}+6}}E_{m,i'}^{k,j,\mathscr{l}}(t)\frac{dt}{t}\leq C_N2^{-\mathscr{l}(N-1)}2^{(j_0-j)N}2^{k}\delta^2\lVert \chi_{\widetilde{Q_m}}\psi_k(\sqrt{\mathcal{L}^{\alpha}})f\rVert_2^2.
\end{equation*}
Now we take summation over $\mathscr{l}$ and $m\in M_{0}$ in \eqref{B18} to get
\begin{equation*}
I_1(j,k)\leq C_N 2^{j_0\beta/2}2^{(j_0-j)(N-\beta)/2}2^{k(1-\beta)/2}\delta\left(2^{(k-j)\beta}\sum_{m\in M_0}\lVert \chi_{\widetilde{Q_m}}\psi_k(\sqrt{\mathcal{L}^{\alpha}})f\rVert_2^2\right)^{1/2}.
\end{equation*}
Since $j>k$ and $m\in M_0$, we note that $(1+\arrowvert x\arrowvert)^{\beta}\leq C2^{(j-k)\beta}$ if $x\in \widetilde {Q_m}$. It follows that
\begin{equation}\label{BB}
2^{(k-j)\beta}\sum_{m\in M_0}\lVert \chi_{\widetilde{Q_m}}\psi_k(\sqrt{\mathcal{L}^{\alpha}})f\rVert_2^2\leq C\int_{\mathbb{R}_{+}^{d}}\arrowvert \psi_k(\sqrt{\mathcal{L}^{\alpha}})f(x)\arrowvert ^2(1+\arrowvert x\arrowvert)^{-\beta}dx.
\end{equation}
Noting that $j_0=[-\log_2\delta]-1$ and $k\geq [-\frac{1}{2}\log_2\delta]$, we obtain
\begin{equation*}
I_1(j,k)\leq C_N 2^{(j_0-j)(N-\beta)/2}\delta^{3/4-\beta/4}\left(\int_{\mathbb{R}_{+}^{d}}\arrowvert \psi_k(\sqrt{\mathcal{L}^{\alpha}})f(x)\arrowvert ^2(1+\arrowvert x\arrowvert)^{-\beta}dx\right)^{1/2},
\end{equation*}
which gives \eqref{B21} since $N>0$ is arbitrary. We now proceed to prove \eqref{B22}. Note that supp\,$\varphi_{\mathscr{l},\delta}\subseteq (1-2^{\mathscr{l}+2}\delta,1+2^{\mathscr{l}+2}\delta)$, and so supp\,$(\varphi_{\mathscr{l},\delta}\phi_{\delta,j})(\cdot/t)\subset [t(1-2^{\mathscr{l}+2}\delta), t(2^{\mathscr{l}+2}\delta)]$. Thus, setting $R=[t(1+2^{\mathscr{l}+2}\delta)]$, by \eqref{ETrace1} we get
\begin{equation}\label{B23}
E_{m,i'}^{k,j,\mathscr{l}}(t)\leq R\lVert (\varphi_{\mathscr{l},\delta}\phi_{\delta,j})(R\cdot/t)\rVert_{R^2,2}^2\int_{\mathbb{R}_{+}^{d}}\arrowvert \eta_{i'}(\sqrt{\mathcal{L}^{\alpha}})[ \chi_{\widetilde{Q_m}}\psi_k(\sqrt{\mathcal{L}^{\alpha}})f](x)\arrowvert ^2dx
\end{equation}
for $0\leq \mathscr{l}\leq [-\log_2\delta]-3$. By formula (3.34) in \cite{CD}, we have
\begin{equation}\label{BBB}
\lVert (\varphi_{\mathscr{l},\delta}\phi_{\delta,j})(R\cdot/t)\rVert_{R^2,2}\leq\lVert (\varphi_{\mathscr{l},\delta}\phi_{\delta,j})((1+2^{\mathscr{l}+2}\delta)\cdot)\rVert_{\infty}(2^{\mathscr{l}+3}\delta)^{1/2}\leq C_N2^{(j_0-j)N}2^{-\mathscr{l}N}(2^{\mathscr{l}+3}\delta)^{1/2}
\end{equation}
for any $N>0$ and $\mathscr{l}\geq 0$. Since $R\sim 2^{k}$, combining this with \eqref{B23}, we get the desired \eqref{B22}.
\par For the case $d=1$, it can be handled in the same manner as before, the only difference is that we use \eqref{ETrace2} instead of \eqref{ETrace1}, so we omit those steps.\\
\\
\textit{Estimate of the term  $I_2(j,k)$.} We can deduce as before that
\begin{equation*}
E_{m,i'}^{k,j,\mathscr{l}}(t)\leq \lVert (\varphi_{\mathscr{l},\delta}\phi_{\delta,j})(\cdot/t)\rVert_{\infty}^2\int_{\mathbb{R}_{+}^{d}}\arrowvert \eta_{i'}(\sqrt{\mathcal{L}^{\alpha}})[ \chi_{\widetilde{Q_m}}\psi_k(\sqrt{\mathcal{L}^{\alpha}})f](x)\arrowvert ^2dx.
\end{equation*}
From the definition of $\varphi_{\mathscr{l},\delta}$ and \eqref{B4} we have
\begin{equation*}
\lVert (\varphi_{\mathscr{l},\delta}\phi_{\delta,j})(\cdot/t)\rVert_{\infty}\leq C_N2^{j-j_0}(2^{j+\mathscr{l}}\delta)^{-N},\quad \mathscr{l}\geq[-\log_2\delta]-2.
\end{equation*}
Thus, it follows that
\begin{equation*}
E_{m,i'}^{k,j,\mathscr{l}}(t)\leq C_N2^{2(j-j_0)}(2^{j+\mathscr{l}}\delta)^{-2N}\int_{\mathbb{R}_{+}^{d}}\arrowvert \eta_{i'}(\sqrt{\mathcal{L}^{\alpha}})[ \chi_{\widetilde{Q_m}}\psi_k(\sqrt{\mathcal{L}^{\alpha}})f](x)\arrowvert ^2dx.
\end{equation*}
After putting this in \eqref{B19}, we take summation over $m\in M_{0}$ to obtain
\begin{equation*}
I_2(j,k)\leq C_N\delta^{1/2-N}2^{(j-j_0)}2^{-N_j}2^{(j-k)\alpha/2}\sum_{\mathscr{l}=[-\log_2\delta]-2}^{\infty}2^{-\mathscr{l}(N-2)}\left(2^{(k-j)\beta}\sum_{m\in M_0}\lVert \chi_{\widetilde{Q_m}}\psi_k(\sqrt{\mathcal{L}^{\alpha}})f\rVert_2^2\right)^{1/2}.
\end{equation*}
Since $j>k$, $j_0=[-\log_2\delta]-1$ and $k\geq[-\frac{1}{2}\log_2\delta]$, using \eqref{BB} and taking sum over $\mathscr{l}$ we obtain
\begin{equation}\label{B24}
I_2(j,k)\leq C_N2^{(j_0-j)N}\delta^N \left(\int_{\mathbb{R}_{+}^{d}}\arrowvert \psi_k(\sqrt{\mathcal{L}^{\alpha}})f(x)\arrowvert ^2(1+\arrowvert x\arrowvert)^{-\beta}dx\right)^{1/2}
\end{equation}
for any $N>0$.\\
\\
\textit{Estimate of the term  $I_3(j,k)$.} We now prove that
\begin{equation}\label{B25}
I_3(j,k)\leq C_N2^{(j_0-j)N}\delta^{1/2} \left(\int_{\mathbb{R}_{+}^{d}}\arrowvert \psi_k(\sqrt{\mathcal{L}^{\alpha}})f(x)\arrowvert ^2(1+\arrowvert x\arrowvert)^{-\beta}dx\right)^{1/2}.
\end{equation}
Recall that $x_m=2^{j-k+2}m$ is the center of the cube $Q_m$. We begin with making an observation that
\begin{equation}\label{B26}
C^{-1}(1+\arrowvert x_m\arrowvert)\leq 1+\arrowvert x\arrowvert\leq C(1+\arrowvert x_m\arrowvert),\quad x\in Q_m
\end{equation}
provided that $m\notin M_{0}$. By this observation, we can deduce that 
\begin{equation*}
E_{m,i'}^{k,j,\mathscr{l}}(t)\leq (1+\arrowvert x_m\arrowvert)^{-\beta}\lVert (\varphi_{\mathscr{l},\delta}\phi_{\delta,j})(t^{-1}\sqrt{\mathcal{L}^{\alpha}})\rVert_{2\rightarrow 2}^{2}\lVert  \eta_{i'}(\sqrt{\mathcal{L}^{\alpha}})[ \chi_{\widetilde{Q_m}}\psi_k(\sqrt{\mathcal{L}^{\alpha}})f]\rVert_2^2.
\end{equation*}
Since $\lVert (\varphi_{\mathscr{l},\delta}\phi_{\delta,j})(t^{-1}\sqrt{\mathcal{L}^{\alpha}})\rVert_{2\rightarrow 2}\leq \lVert (\varphi_{\mathscr{l},\delta}\phi_{\delta,j})\rVert_{\infty}$, it follows from \eqref{BBB} that we have
\begin{equation*}
E_{m,i'}^{k,j,\mathscr{l}}(t)\leq (1+\arrowvert x_m\arrowvert)^{-\beta}2^{2(j_0-j)N}2^{-2\mathscr{l}N}\lVert  \eta_{i'}(\sqrt{\mathcal{L}^{\alpha}})[ \chi_{\widetilde{Q_m}}\psi_k(\sqrt{\mathcal{L}^{\alpha}})f]\rVert_2^2.
\end{equation*}
Using this and disjointness of the spectral supports, we obtain
\begin{equation*}
\begin{aligned}
\sum_{m\notin M_{0}}\sum_{i}\int_{I_i}\sum_{i'=i-2^{\mathscr{l}+6}}^{i+2^{\mathscr{l}+6}}E_{m,i'}^{k,j,\mathscr{l}}(t)\frac{dt}{t}&\leq C2^{2(j_0-j)N}2^{-2\mathscr{l}N}\delta \sum_{m\notin M_{0}}(1+\arrowvert x_m\arrowvert)^{-\beta}\lVert \chi_{\widetilde{Q_m}}\psi_k(\sqrt{\mathcal{L}^{\alpha}})f\rVert_2^2\\
&\leq C2^{2(j_0-j)N}2^{-2\mathscr{l}N}\delta \sum_{m}\int_{\widetilde{Q_m}}\arrowvert \chi_{\widetilde{Q_m}}\psi_k(\sqrt{\mathcal{L}^{\alpha}})f\arrowvert^2 (1+\arrowvert x\arrowvert)^{-\beta}dx\\
&\leq C2^{2(j_0-j)N}2^{-2\mathscr{l}N}\delta \int_{\mathbb{R}_+^d}\arrowvert \psi_k(\sqrt{\mathcal{L}^{\alpha}})f\arrowvert^2 (1+\arrowvert x\arrowvert)^{-\beta}dx.
\end{aligned}
\end{equation*}
Finally, recalling \eqref{B20} and taking sum over $\mathscr{l}$ yields the estimate \eqref{B25}.
\par Therefore, since $\delta\leq \delta A_{\beta,d}^{\varepsilon}(\delta)$, we thus deduce from \eqref{B21}, \eqref{B25} and \eqref{B25} that
\begin{equation}\label{B27}
\int_{2^{k-1}}^{2^{k+2}}E^{k,j}(t)\frac{dt}{t}\leq C_N2^{2(j_0-j)N} \delta A_{\beta,d}^{\varepsilon}(\delta)\int_{\mathbb{R}_+^d}\arrowvert \psi_k(\sqrt{\mathcal{L}^{\alpha}})f\arrowvert^2 (1+\arrowvert x\arrowvert)^{-\beta}dx
\end{equation}
for any $N>0$ if $j>k $.
\\
\par 
\textbf{Case 2: $j \leq k$.}
In this case, the side length of each $Q_m$ is less than $c_0 4$, so \eqref{B26} holds for any $m \in \mathbb{N}^d$. Thus, the argument used in proving \eqref{B25} applies without modification. We deduce that
\begin{equation}\label{B28}
\int_{2^{k-1}}^{2^{k+2}} E^{k,j}(t) \frac{dt}{t} \leq C_N 2^{2(j_0 - j)N} \delta \int_{\mathbb{R}_+^d} \left\lvert \psi_k(\sqrt{\mathcal{L}^\alpha}) f \right\rvert^2 (1 + |x|)^{-\beta}  dx
\end{equation}
for any $N > 0$.

\textit{Completion of the proof of \eqref{high}.}
By \eqref{B27}, \eqref{B28}, and the inequality $\delta \leq \delta A_{\beta,d}^{\varepsilon}(\delta)$, we have established \eqref{B27} for all $j \geq j_0$ and $k$. Substituting \eqref{B27} into the right-hand side of \eqref{B9} and summing over $j$, we obtain
\begin{equation}
\int_{\mathbb{R}_+^d} \left\lvert \mathfrak{S}_{\delta}^{\ell} f(x) \right\rvert^2 (1 + |x|)^{-\beta}  dx \leq C \delta A_{\beta,d}^{\varepsilon}(\delta) \int_{\mathbb{R}_+^d} \sum_{k} \left\lvert \psi_k(\sqrt{\mathcal{L}^\alpha}) f \right\rvert^2 (1 + |x|)^{-\beta}  dx.
\end{equation}
Applying the Littlewood-Paley inequality \eqref{PL1.1} yields the estimate \eqref{high}, completing the proof of Lemma~\ref{SL}.
\section{Sharpness of the summability indices}\label{Section4}
In this section we discuss sharpness of summability indices given in Theorem \ref{T1.1}. The following proves the necessity parts of Theorem \ref{T1.1}.
\begin{proposition}\label{P4.1}
Let $d\geq 2$, and $p>2d/(d-1)$.If $0\leq \delta<\lambda(p)/2$, then there exists a measurable function $f$ such that $f\in L^{p}(\mathbb{R}_{+}^{d})$ and \eqref{Fail} holds. 
\end{proposition}
To prove Proposition \ref{P4.1}, we construct a sequence of functions that behave as if they were the eigenfunctions of $\mathcal{L}^{\alpha}$ on the set $\mathbb{E}=\{x\in \mathbb{R}_{+}^d: 1\leq\arrowvert x\arrowvert\leq 2\}$. For this purpose, we will utilizing the generating formula \eqref{Generate} of the Laguerre polynomials. For every one-dimensional Laguerre polynomial $L_{n}^{\alpha}$, we define
\begin{equation*}
\tilde{L}_{n}^{\alpha}(x)=L_{n}^{\alpha}(x)/\sqrt{L_{n}^{\alpha}(0)},
\end{equation*}
where 
\begin{equation*}
L_{n}^{\alpha}(0)=\frac{\Gamma(n+\alpha+1)}{\Gamma(n+1)\Gamma(\alpha+1)}.
\end{equation*}
 For a multi-index  $\alpha=(\alpha_1,\alpha_2,\cdot\cdot\cdot,\alpha_d)$, $\alpha_i>-1$, the multiple Laguerre polynomials are defined by the formula
 \begin{equation*}
 \tilde{L}_{\mu}^{\alpha}(x)=\prod_{i=1}^{d} \tilde{L}_{\mu_i}^{\alpha_i}(x_i), \quad \mu\in \mathbb{N}^d.
 \end{equation*}
 The generating formula \eqref{Generate} of the Laguerre polynomials can be written as
 \begin{equation*}
 \sum_{n=0}^{\infty}\tilde{L}_{n}^{\alpha}(0)\tilde{L}_{n}^{\alpha}(x)\omega^n=(1-\omega)^{-\alpha-1}e^{-x\omega/(1-\omega)}, \quad \arrowvert \omega\arrowvert<1.
 \end{equation*}
 Multiplying the above formula leads to (see also \cite{XY}),
 \begin{equation}\label{4.6}
 \sum_{n=0}^{\infty}\sum_{\arrowvert \mu\arrowvert_1=n}\tilde{L}_{\mu}^{\alpha}(0)\tilde{L}_{\mu}^{\alpha}(x)\omega^n=(1-\omega)^{-\arrowvert\alpha\arrowvert_1-d}e^{-\arrowvert x\arrowvert_1 \omega/(1-\omega)}, \quad \arrowvert \omega\arrowvert<1,
 \end{equation}
 which implies that $\sum_{\arrowvert \mu\arrowvert_1=n}\tilde{L}_{\mu}^{\alpha}(0)\tilde{L}_{\mu}^{\alpha}(x)=L_n^{\arrowvert\alpha\arrowvert_1+d-1}(\arrowvert x\arrowvert_1)$ (see also, for example, \cite[Proof of Theorem 2.1]{XY}).
For every $n\in\mathbb{N}$, we define function
\begin{equation*}
G_{n}(x)=\sum_{\arrowvert \mu\arrowvert_1=n}\varphi_{\mu}^{\alpha}(x)\tilde{L}_{\mu}^{\alpha}(0).
\end{equation*}
From \eqref{L1} and \eqref{Ld}, we obtain
\begin{equation}\label{4.8}
\begin{aligned}
G_{n}(x)&=\left(e^{-\frac{\arrowvert x\arrowvert^2}{2}}\prod_{i=1}^{d}\sqrt{\frac{2}{\Gamma(\alpha_i+1)}}x_i^{\alpha_i+1/2}\right)\sum_{\arrowvert \mu\arrowvert=n}\tilde{L}_{\mu}^{\alpha}(x^2)\tilde{L}_{\mu}^{\alpha}(0)\\
&=e^{-\frac{\arrowvert x\arrowvert^2}{2}}\prod_{i=1}^{d}\sqrt{\frac{2}{\Gamma(\alpha_i+1)}}x_i^{\alpha_i+1/2}L_n^{\arrowvert\alpha\arrowvert_1+d-1}(\arrowvert x\arrowvert^2),
\end{aligned}
\end{equation}
where $x^2=(x_1^2,\cdot\cdot\cdot,x_d^2)$.
Putting \eqref{4.8} into \eqref{4.6}, we obtain
\begin{equation*}
\sum_{n=0}^{\infty}G_n(x)\omega^n= e^{-\frac{\arrowvert x\arrowvert^2}{2}}(1-\omega)^{-\arrowvert\alpha\arrowvert_1-d}e^{-\arrowvert x\arrowvert^2 \omega/(1-\omega)}\prod_{i=1}^{d}\sqrt{\frac{2}{\Gamma(\alpha_i+1)}}x_i^{\alpha_i+1/2}.
\end{equation*}
For $0\leq r<1$, and $t\in\mathbb{R}$, we set $\omega=re^{-it}$, and define function $G_{r}(t,x)$ by the series
\begin{equation}\label{4.9}
G_{r}(t,x)=\sum_{n=0}^{\infty}G_n(x)r^ne^{-nit}=e^{-\frac{\arrowvert x\arrowvert^2}{2}}(1-re^{-it})^{-\arrowvert\alpha\arrowvert_1-d}e^{-\arrowvert x\arrowvert^2 re^{-it}/(1-re^{-it})}\prod_{i=1}^{d}\sqrt{\frac{2}{\Gamma(\alpha_i+1)}}x_i^{\alpha_i+1/2}.
\end{equation}
 It is straightforward to see that the limit of $G_{r}(t,x)$ as $r$ tends to 1 exists whenever $\sin \frac{t}{2}$ is different from zero, we denote the limit by $G(t,x)$ (The derivation of this function follows the same procedure as that of the kernel of the Schrödinger propagator $e^{it\mathcal{L}^{\alpha}}$, see, for example, \cite{TS,TS4}) and we have
\begin{equation}\label{4.13}
G(t,x)=\left(\frac{e^{it/2}}{2i\sin \frac{t}{2}}\right)^{\arrowvert \alpha\arrowvert_1+d}e^{\frac{i}{2}\arrowvert x\arrowvert^2\cot \frac{t}{2}}\prod_{i=1}^{d}\sqrt{\frac{2}{\Gamma(\alpha_i+1)}}x_i^{\alpha_i+\frac{1}{2}}.
\end{equation}
Using \eqref{4.13}, we have following result.
\begin{lemma}\label{Lemma4.2}
Let $p>2d/(d-1)$. Then, there are sequences $\{\mu_k\}\subset \mathbb{N}$ and $\{f_k\}\subset\mathcal{S}(\mathbb{R}^{d})$ such that 
\begin{equation*}
\mu_k\sim 2^{2^{k}},\quad\lVert f_k\rVert_{L^p(\mathbb{R}_{+}^d)}=1,
\end{equation*}
and the following holds for a large constant $k_0$:
\begin{equation}\label{L4.11}
\left\arrowvert\left\{x\in \mathbb{E}:\arrowvert P_{\mu_k}f_k(x)\arrowvert\geq C_0\mu_k^{\lambda(p)/2}\right\}\right\arrowvert\geq C_0
\end{equation}
for a constant $C_0>0$ if $k\geq k_0$, and for any $N>0$ there is a constant $C_N>0$ such that
\begin{equation}\label{L4.12}
\arrowvert P_{\mu_k}f_j(x)\arrowvert\leq C_N\mu_k^{\lambda(p)/2}(\mu_j/\mu_k)^{\frac{1}{2}-\frac{\arrowvert\alpha\arrowvert}{2}-\frac{d}{4}-\frac{d}{2p}}\arrowvert \mu_k-\mu_j\arrowvert^{-N}, \quad x\in\mathbb{E}
\end{equation}
whenever $j\neq k\geq k_0$.
\end{lemma}
\begin{proof}
Based the argument as \cite[Proof of Lemma 5.2]{JLR}. We take a sequence $\{\mu_k\}\subset \mathbb{N}$ such that $\mu_k\sim 2^{2^{k}}$. Set
\begin{equation*}
\begin{aligned}
g_k(x)&=\int\phi(t)\lim_{r\rightarrow 1}G_r(t,x)e^{-\mu_kti}dt\\
&=\left(\prod_{i=1}^{d}\sqrt{\frac{2}{\Gamma(\alpha_i+1)}}x_i^{\alpha_i+\frac{1}{2}}\right)\int\phi(t)\left(\frac{e^{it/2}}{2i\sin \frac{t}{2}}\right)^{\arrowvert \alpha\arrowvert+d}e^{i\left(\frac{i}{2}\arrowvert x\arrowvert^2\cot \frac{t}{2}-\mu_kt\right)}dt,\quad x\in \mathbb{R}_{+}^{d}.
\end{aligned}
\end{equation*}
where $\phi\in C_{c}^{\infty}((2^{-3},2^{-1}))$. Since the oscillatory integral 
\begin{equation*}
\int\phi(t)\left(\frac{e^{it/2}}{2i\sin \frac{t}{2}}\right)^{\arrowvert \alpha\arrowvert+d}e^{i\left(\frac{i}{2}\arrowvert x\arrowvert^2\cot \frac{t}{2}-\mu_kt\right)}dt
\end{equation*}
obeys the following estimates (See \cite[Proof of Lemma 5.2]{JLR} for the detail).
\begin{equation}\label{4.14}
\left\arrowvert\int\phi(t)\left(\frac{e^{it/2}}{2i\sin \frac{t}{2}}\right)^{\arrowvert \alpha\arrowvert+d}e^{i\left(\frac{i}{2}\arrowvert x\arrowvert^2\cot \frac{t}{2}-\mu_kt\right)}dt\right\arrowvert\leq\left\{\begin{array}{lr}
C\mu_k^{-1/2},\quad\qquad\quad\quad\quad\qquad\, \,\,\,if\,\,\arrowvert x\arrowvert^2\sim\mu_k,\\
C_N(1+\max\{\arrowvert x\arrowvert^2,\mu_k\})^{-N}, \quad\, otherwise
\end{array}
\right.
\end{equation}
Moreover, the stationary phase method gives (See \cite[Proof of Lemma 5.2]{JLR})
\begin{equation}\label{4.15}
\left\arrowvert\int\phi(t)\left(\frac{e^{it/2}}{2i\sin \frac{t}{2}}\right)^{\arrowvert \alpha\arrowvert+d}e^{i\left(\frac{i}{2}\arrowvert x\arrowvert^2\cot \frac{t}{2}-\mu_kt\right)}dt\right\arrowvert\sim\mu_k^{-1/2}, \,\, if\,\,\arrowvert x\arrowvert^2\sim\mu_k.
\end{equation}
Thus, for $k\geq k_0$ large enough, using the polar coordinate, \eqref{4.14} and \eqref{4.15} gives
\begin{equation}\label{L4.15}
\lVert g_k\rVert_{L^p(\mathbb{R}_+^{d})}\sim_p\mu_k^{(\frac{\arrowvert\alpha\arrowvert_1}{2}+\frac{d}{4})+\frac{d}{2p}-\frac{1}{2}}.
\end{equation}
Note that $P_{\mu_k}g_j(x)=\hat{\phi}(\mu_k-\mu_j)G_{\mu_k}(x)$. Indeed, since $\phi\in C_{c}^{\infty}((2^{-3},2^{-1}))$ and $\arrowvert e^{-\arrowvert x\arrowvert^2 re^{-2it}/(1-re^{-2it})}\arrowvert\leq C$, for $1/2\leq r\leq1$ (see, for example \cite[Proof of Lemma 3.1]{TS1}). Then, by the dominated convergence theorem we have
\begin{equation*}
g_k(x)=\lim_{r\rightarrow 1}\int\phi(t)G_r(t,x)e^{-\mu_kti}dt.
\end{equation*}
Thus, for every eigenfunction $\varphi_{\mu}^{\alpha}(x)$, we have
\begin{equation}\label{4.17}
\begin{aligned}
\langle g_k(x), \varphi_{\mu}^{\alpha}\rangle&=\langle\lim_{r\rightarrow 1}\int\phi(t)G_r(t,x)e^{\mu_kti}dt, \varphi_{\mu}^{\alpha}\rangle\\&
=\lim_{r\rightarrow 1}\langle\int\phi(t)G_r(t,x)e^{\mu_kti}dt, \varphi_{\mu}^{\alpha}\rangle\\
&=\lim_{r\rightarrow 1}\int\phi(t)\langle G_r(t,x),\varphi_{\mu}^{\alpha} \rangle e^{\mu_kti}dt\\
&=\lim_{r\rightarrow 1}\int\phi(t)\tilde{L}_{\mu}^{\alpha}(0)r^{\arrowvert \mu\arrowvert_1}e^{-\arrowvert\mu\arrowvert_1 it}e^{\mu_kti}dt\\
&=\hat{\phi}(\arrowvert\mu\arrowvert_1-\mu_k)\tilde{L}_{\mu}^{\alpha}(0).
\end{aligned}
\end{equation}
The second equality follows from the dominated convergence theorem. In the third equality, we use the Fubini's theorem. The fourth equality follows from the expansion \eqref{4.9}. Clearly, \eqref{4.17} implies that $P_{\mu_k}g_j(x)=\hat{\phi}(\mu_k-\mu_j)G_{\mu_k}(x)$.
Recall \eqref{4.8} and \eqref{Define}, we obtain that
\begin{equation*}
G_{\mu_k}(x)=\sqrt{\frac{\Gamma(\mu_k+\arrowvert \alpha\arrowvert_1+d)}{\Gamma(\mu_k+1)}}\arrowvert x\arrowvert^{1-\arrowvert \alpha\arrowvert_1-d}\mathscr{L}_{\mu_k}^{\arrowvert\alpha\arrowvert_1+d-1}(\arrowvert x\arrowvert^2)\prod_{i=1}^{d}\sqrt{\frac{2}{\Gamma(\alpha_i+1)}}x_i^{\alpha_i+1/2}.
\end{equation*}
Using \eqref{L1.2} for $a=\arrowvert\alpha\arrowvert_1+d-1$, we see that
\begin{equation}\label{L4.19}
\left\arrowvert\left\{x\in \mathbb{E}: \arrowvert G_{\mu_k}(x)\arrowvert\sim\mu_k^{\frac{\arrowvert\alpha\arrowvert_1+d}{2}-\frac{3}{4}}\right\}\right\arrowvert\geq C_0
\end{equation}
for a constant $C_0>0$ and for $k\geq k_0$ large enough (see, for example, \cite[Proof of Lemma 4.9]{CD}). We set
\begin{equation*}
f_k=g_k/\lVert g_k\lVert_{L^p(\mathbb{R}_+^d)}.
\end{equation*}
It remains to verify \eqref{L4.11} and \eqref{L4.12}. In fact, \eqref{L4.11} follows form \eqref{L4.15} and \eqref{L4.19} since we can take $\hat{\phi}(0)>0$ and $P_{\mu_k}f_j(x)=\hat{\phi}(\mu_k-\mu_j)G_{\mu_k}(x)/\lVert g_j\lVert_{L^p(\mathbb{R}_+^d)}$. Using \eqref{L1.2} for $a=\arrowvert\alpha\arrowvert_1+d-1$, we see that $\arrowvert G_{\mu_k}(x)\arrowvert \lesssim \mu_k^{\frac{\arrowvert\alpha\arrowvert_1+d}{2}-\frac{3}{4}}$ if $\arrowvert x\arrowvert\sim 1$. Combing this and \eqref{L4.15}, we have
\begin{equation*}
\arrowvert P_{\mu_k}f_j(x)\arrowvert\lesssim \mu_{k}^{\lambda(p)/2}(\mu_j/\mu_k)^{\frac{1}{2}-(\frac{\arrowvert\alpha\arrowvert_1}{2}-\frac{d}{4})-\frac{d}{2p}}\arrowvert\hat{\phi}(\mu_k-\mu_j)\arrowvert, \quad \arrowvert x\arrowvert\sim 1
\end{equation*}
for k$,j$ large enough. By rapid decay of $\hat{\phi}$, this give \eqref{L4.12}.
\end{proof}
With Lemma \ref{Lemma4.2} in hands, we now prove Proposition \ref{P4.1}.\\
\\
\textbf{Proof of Proposition \ref{P4.1}.} Let $\{\mu_k\}$ and $\{f_k\}$ be the sequences given in Lemma \ref{Lemma4.2}. We consider $f=\sum_{k=k_0}^{\infty}2^{-k}f_k$ and 
\begin{equation*}
E_k=\left\{x\in \mathbb{E}:\arrowvert S_{*}^{\lambda}(\mathcal{L}^{\alpha})f(x)\arrowvert\geq c2^{-k}\mu_k^{-\lambda+\lambda(p)/2}\right\},\quad k\geq k_0
\end{equation*}
for a small positive constant $c$ to be chosen later. Since $\mu_k\sim 2^{2^{k}}$ and $\lambda< \lambda(p)/2$, then $E_k$ is a decreasing sequence of measurable sets which converges to $E:=\left\{x\in \mathbb{E}: S_{*}^{\lambda}(\mathcal{L}^{\alpha})f(x)=\infty\right\}$. Thus, to prove \eqref{Fail}, it is therefore sufficient to show
\begin{equation}\label{pp4.1}
\arrowvert E_k\arrowvert\geq C_0
\end{equation}
for a constant $C_0>0$ if $k\geq k_0$. However, the inequality \eqref{pp4.1} is an easy consequence of \eqref{L4.11} and the inclusion relation
\begin{equation}\label{pp5.1}
\tilde{E}_{k}:=\left\{x\in \mathbb{E}:\arrowvert P_{\mu_k}f_k(x)\arrowvert\geq C_0\mu_k^{\lambda(p)/2}\right\}\subset E_k.
\end{equation}
Hence, it is enough to show \eqref{pp5.1}. To this end, we invoke the inequality
\begin{equation}\label{pp6.1}
P_{\mu_k}f_k(x)\leq C_*\mu_k^{\lambda}S_{*}^{\lambda}(\mathcal{L}^{\alpha})f(x), \quad k\geq k_0,
\end{equation}
which holds with a constant $C_*$. See \cite[5.5]{JLR} for the detail. By \eqref{pp6.1}, we obtain that
\begin{equation}\label{pp7.1}
S_{*}^{\lambda}(\mathcal{L}^{\alpha})f(x)\geq C^{-1}\mu_k^{-\lambda}\left(2^{-k}\arrowvert P_{\mu_k}f_k(x)\arrowvert-\sum_{j\neq k}2^{-j}\arrowvert P_{\mu_k}f_j(x)\arrowvert\right).
\end{equation}
Since $\mu_k\sim 2^{2^{k}}$, using \eqref{L4.12} with sufficiently large $N$, we see that $\sum_{j\neq k}2^{-j}\arrowvert P_{\mu_k}f_j(x)\arrowvert$ is bounded above by a constant times
\begin{equation*}
\mu^{\lambda(p)/2}\sum_{j\neq k}2^{-j}(\mu_j/\mu_k)^{\frac{1}{2}-\frac{\arrowvert\alpha\arrowvert}{2}-\frac{d}{4}-\frac{d}{2p}}\arrowvert \mu_k-\mu_j\arrowvert^{-N}\lesssim \mu^{\lambda(p)/2}\mu_k^{-1}.
\end{equation*}
We choose a constant $c$ such that $c<C_0/2C_*$. Using \eqref{pp7.1}, for $x\in \tilde{E}_k$ we have $S_{*}^{\lambda}(\mathcal{L}^{\alpha})f(x)\geq c2^{-k}\mu_k^{-\lambda+\lambda(p)/2}$ if $k$ is large enough. Thus, \eqref{pp5.1} follows.$\hfill{\Box}$
\section*{Acknowledgments}
Z. Duan was supported by the National Natural Science Foundation of China under grants 61671009 and 12171178. 
\section*{Declarations}
\textbf{Conflict of interest} The authors declare that there is no conflict of interest.

\end{document}